\documentclass[ejs,twoside,preprint]{imsart}
\usepackage[OT1]{fontenc}
\usepackage{amsthm,amsmath,amsfonts,amssymb}
\usepackage[all,cmtip]{xy}
\arxiv{math.ST/0904.1980}
\usepackage{graphicx} 
\usepackage[colorlinks]{hyperref}

\doi{10.1214/11-EJS640}
\pubyear{2011}
\volume{5}
\issue{0}
\firstpage{1276}
\lastpage{1312}

\startlocaldefs
\def\@pnumwidth{22pt}

\theoremstyle{plain}
\newtheorem{thm}{Theorem}[section]
\newtheorem{lem}[thm]{Lemma}
\newtheorem{prop}[thm]{Proposition}
\newtheorem{cor}[thm]{Corollary}

\theoremstyle{definition}
\newtheorem{defn}[thm]{Definition}

\newtheorem{exmp}[thm]{Example}

\theoremstyle{remark}
\newtheorem{rem}[thm]{Remark}

\newcommand{\indep}{{\;\bot\!\!\!\!\!\!\bot\;}}

\endlocaldefs

\begin{document}
\begin{frontmatter}
\title{Implicit inequality constraints in a binary tree model}
\runtitle{Geometry of the binary models on trees}

\begin{aug}
\author{\fnms{Piotr} \snm{Zwiernik}\corref{}\ead[label=e1]{piotr.zwiernik@gmail.com}}

\address{Institute for Pure and Applied Mathematics,\\
460 Portola Plaza,\\ Box 957121,\\ Los Angeles, CA 90095-7121\\
\printead{e1}}
\end{aug}

\medskip
\textbf{\and}
\begin{aug}
\author{\fnms{Jim Q.} \snm{Smith}\ead[label=e3]{j.q.smith@warwick.ac.uk}}
\address{University of Warwick\\Department of Statistics\\CV7AL,
Coventry, UK\\
\printead{e3}}

\runauthor{P. Zwiernik and J.Q. Smith}

\affiliation{University of Warwick}
\end{aug}

\begin{abstract}
In this paper we investigate the geometry of a discrete Bayesian
network whose graph is a tree all of whose variables are binary and the
only observed variables are those labeling its leaves. We provide the
full geometric description of these models which is given by a set of
polynomial equations together with a set of complementary implied
inequalities induced by the positivity of probabilities on hidden
variables. The phylogenetic invariants given by the equations can be
useful in the construction of simple diagnostic tests. However, in this
paper we point out the importance of also incorporating the associated
inequalities into any statistical analysis. The full characterization
of these inequality constraints derived in this paper helps us
determine how and why routine statistical methods can break down for this
model class.
\end{abstract}

\begin{keyword}[class=AMS]
\kwd[Primary ]{62H05}\kwd{62E15}
\kwd[; secondary ]{60K99}\kwd{62F99}
\end{keyword}

\begin{keyword}
\kwd{Graphical models on trees}
\kwd{binary data}
\kwd{tree cumulants}
\kwd{semialgebraic statistical models}
\kwd{phylogenetic invariants}
\kwd{inequality constraints}
\end{keyword}
\received{\smonth{10} \syear{2010}}

\tableofcontents

\end{frontmatter}

\section{Introduction}\label{sec:introduction}

A Bayesian network whose graph is a tree all of whose inner nodes
represent variables which are not directly observed defines an
important class of models containing both phylogenetic tree models and
hidden Markov models. Inference for this model class tends to be
challenging and often needs to employ fragile numerical algorithms. In
\cite{pwz-2010-identifiability} we established a useful new coordinate
system to analyze such models when all of the variables are binary.
This reparametrization enabled us not only to address various
identifiability issues but also helped us to derive exact formulae for
the maximum likelihood estimators given that the sample proportions
were in this model class.

However, as well as making identifiability issues more transparent and
open to systematic analysis, this new coordinate system can be also
used to analyze the global structure of tree models. In particular, it
enables us to obtain the full description of these models in terms of
implicit polynomial equations and inequalities. Knowing this full
semi-algebraic description is extremely useful when used in conjunction
with the identifiability structure as discussed in
\cite{pwz-2010-identifiability}. We explain in Section
\ref{sec:inference} how this study impacts the stability of the maximum
likelihood and Bayesian estimation procedures within the class of
phylogenetic tree models. It is also helpful in the construction of
tree diagnostics and model selection procedures within this class.

This paper builds on the results in \cite{chor2000mml} where some
partial understanding of the analytic approach to the maximum
likelihood estimation was presented. The problem here is that routinely
fitted phylogenetic models often violate the inequality constraints
defining the model. One effect of this phenomenon is then that the
maximum likelihood estimators (MLEs) usually lie on to the boundaries
of the parameter space (see Section \ref{sec:inference} for an
example). In a full Bayesian analysis it will make the ensuing
inference about probabilities highly sensitive to the settings of prior
distributions on the parameters (see \cite{smith2010robustness,smith2008isoseparation}).
This, in turn, automatically interferes
with the appropriate functioning of model selection algorithms. For
example Bayes Factor scores will be highly influenced again by priors.
On the other hand more classical methods like for example AIC or BIC
algorithms, when used routinely, misbehave because many of the MLEs
will lie on the boundary of the feasible region since usual dimension
counting penalties are implicitly too large (see \cite{pwz-2010-bic}).
For these and other reasons explained in more detail in Section
\ref{sec:inference}, the inequality conditions are of considerable
practical importance.\looseness=-1

This paper is part of an explosion of work which apply techniques in
algebraic geometry to study and develop statistical methodologies. The
particular geometric study of tree models was first introduced by Lake
\cite{lake1987rit}, and Cavender and Felsenstein
\cite{cavender1987ips}. This research was initially focused on so
called phylogenetic invariants. These are algebraic relationships
expressed as a set of polynomial equations over the observed
probability tables which must hold for a given phylogenetic model to be
valid. We note that these algebraic techniques have also been embraced
by computational algebraic geometers
\cite{allman2008pia,eriksson2004pag,sturmfels2005tip}
enhancing statistical and computational analysis of such models
\cite{casanellas2007pni} (see also \cite{allman_invariants2007} and
references therein).\looseness=-1

The main technical deficiency of using phylogenetic invariants alone
in this way is that they do not give a \textit{full} geometric
description of the statistical model. However, the additional
inequalities obtained as the main result of this paper complete this
description. Where and how these inequality constraints can helpfully
supplement an analysis based on phylogenetic invariants is illustrated
by the simple example given below.
\begin{exmp}\label{ex:1}
Let $T$ be the tripod tree in Figure \ref{fig:tripod} where we use the
convention that observed nodes are depicted by black nodes.
%
\begin{figure}
\includegraphics{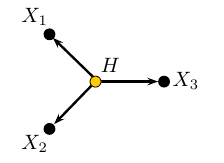}
\caption{The graphical representation of the tripod tree model.}\label
{fig:tripod}
\end{figure}
The inner node represents a binary hidden variable $H$ and the leaves
represent binary observable variables $X_1, X_2, X_3$. The model is
given by all probability distributions $p_{\alpha}$ for
$\alpha\in\{0,1\}^{3}$ such that
\[
p_{\alpha}=\theta^{(H)}_{0}\prod_{i=1}^{3}\theta^{(i)}_{\alpha
_{i}|0}+\theta^{(H)}_{1}\prod_{i=1}^{3}\theta^{(i)}_{\alpha_{i}|1},
\]
where $\theta^{(H)}_{i}=\mathbb{P}(H=i)$ for $i=0,1$ and
$\theta^{{(i)}}_{j|k}=\mathbb{P}(X_{i}=j|H=k)$ for $i=1,2,3$ and
$j,k=0,1$. The model has full dimension over the space of observed
marginal distributions $(X_{1},X_{2},X_{3})$ and consequently there are
no non-trivial equalities defining it. However, it is not a saturated
model since not all the marginal probability distributions over the
observed vector $(X_{1},X_{2},X_{3})$ lie in the model class. For
example Lazarsfeld and Henry \cite[Section 3.1]{lazarsfeld1968lsa}
showed that the second order moments of the observed distribution must
satisfy
\[
{\rm Cov}(X_1,X_2){\rm Cov}(X_1,X_3){\rm Cov}(X_2,X_3)\geq
0.
\]
Together with many other constraints we derive later, this
constraint, which clearly impacts the inferences we might want to make
(see Section \ref{sec:inference}), is not acknowledged through the
study of phylogenetic invariants. Therefore inference based solely on
these invariants is incomplete. For example naive estimates derived
through these methods can be infeasible within the model class in a
sense illustrated later in this paper.
\end{exmp}

This example and the discussion of some inferential issues discussed
above motivated the closer investigation of the semi-algebraic features
associated with the geometry of binary tree models with hidden inner
nodes. The main problem with the geometric analysis of these models is
that, in general, it is hard to obtain all the inequality constraints
defining a model explicitly even for very simple examples (see
\cite[Section 4.3]{drton2007asm}, \cite[Section 7]{garcia2005agb}).
Despite this, some results can be found in the literature. A binary
naive Bayes model was studied by Auvray et al. \cite{auvray2006sad}.
There are also some partial results for general tree structures on
binary variables given by Pearl and Tarsi \cite{pearl_tarsi86} and
Steel and Faller \cite{steel2009mls}. The most important applications
in biology involve variables that can take four values. Recently Matsen
\cite{matsen2007fti} gave a set of inequalities in this case for
group-based phylogenetic models (additional symmetries are assumed)
using the Fourier transformation of the raw probabilities. Here we
provide a simpler and more statistically transparent way to express the
constrained space.

The semialgebraic description we obtain here also has an elegant
mathematical structure. For example \cite{cavender1997443} gave an
intriguing correspondence between, on the one hand, a correlation
system on tree models and on the other distances induced by trees where
the length between two nodes in a tree is given as a sum of the length
of edges in the path joining them. The new coordinate system for tree
models that we introduced in \cite{pwz-2010-identifiability} enables us
to explore in detail this relationship between probabilistic tree
models (also called the tree decomposable distributions in
\cite{pearl_tarsi86}) and tree metrics and extend these results.

It has been known for some time that the constraints on possible
distances between any two leaves in the tree imply some additional
inequality constraints on the possible covariances between the binary
variables represented by the leaves. These inequalities, given in
(\ref{eq:suffic-ineq2}), follow from the four-point condition
(\cite{semple2003pol}, Definition 7.1.5) together with some other
simple non-negativity constraints. By using our new parametrization we
are able to show in this paper that these two types of inequality
constraints cannot be sufficient to describe the model class. Thus any
probability distribution in the model class must satisfy many other
additional constraints involving higher order moments. Using our
methods we are able to provide the full set of the defining constraints
in Theorem \ref{th:parameters0}. This is given by a list of polynomial
equations and inequalities which describe the set of all probability
distributions in the model.

The paper is organized as follows. In Section \ref{sec:models_trees} we
briefly introduce general Markov models. We then proceed to describe a
convenient new change of coordinates for these models given in
\cite{pwz-2010-identifiability}. In the new coordinate system the
parametrization of the model has an elegant product form. We use this
to obtain the full semi-algebraic description of a simple naive Bayes
model. In Section \ref{sec:inference} we discuss various ways in which
an awareness of these implicit inequalities can enrich a statistical
analysis of this model class. In Section \ref{sec:metrics} we state our
main theorem and illustrate how it can be used. In Section
\ref{sec:quartet} we discuss these results for a simple quartet tree
model.

\section{Tree models and tree cumulants}\label{sec:models_trees}\label
{sec:correlations}

We begin by defining and reviewing a new coordinate system for tree
models and demonstrate how it can be used to provide a better
understanding of this model class. We list the main results from our
previous paper \cite{pwz-2010-identifiability} and link it to the
results presented in the next sections.

Parametrizations based on moments are one way of providing a structured
model a structure more amenable to an algebraic analysis (see
\cite{beerenwinkel2007,drton2008binary}). This approach has
proved particularly effective in the presence of hidden data (see
\cite{settimi2000gma}) since then the analysis of a particular marginal
distributions over a subset of the observed variables can be specified
as a function of the joint moments containing that subset only. On the
other hand when a model class is defined by a set of conditional
independences further insight may be provided by reparametrizing to
other \textit{functions} of these moments to elegantly represent this
additional underlying structure. These functions typically resemble
cumulants.

One useful property of standard cumulants is that joint cumulants
always vanish whenever the random vector under analysis can be split
into two independent subvectors. Here we exploit analogous property
using a reparametrization customized to the topology of a particular
tree. These tree cumulants are introduced in
\cite{pwz-2010-identifiability}. They vanish only if some of the edges
in the defining tree model are missing. This corresponds to the
marginal independence of the leaves of two connected components of the
induced forest. The property follows from a more general result in
\cite[Proposition 4.3]{pwz-2010-cumulants} and partly explains the
elegant product-like structure of the resulting parametrization in
Proposition \ref{prop:monomial}.

In this paper we assume that random variables are binary taking values
either $0$ or $1$. We consider models with \textit{hidden} variables,
i.e. variables whose values are never directly observed. The vector $Y$
has as its components all variables in the graphical model, both those
that are observed and those that are hidden. The subvector of $Y$ of
observed variables is denoted by $X$ and the subvector of {hidden}
variables by $H$. A (\textit{directed}) \textit{tree} $T=(V,E)$, where
$V$ is the set of vertices (or nodes) and $E\subseteq V\times V$ is the
set of edges of $T$, is a connected (\textit{directed}) graph with no
cycles. A \textit{rooted tree} is a directed tree that has one
distinguished vertex called the \textit{root}, denoted by the letter
$r$, and all the edges are directed away from $r$. A rooted tree is
usually denoted by $T^{r}$. For each $v\in V$ by ${\rm pa}(v)$ we
denote the node preceding $v$ in $T^{r}$. In particular ${\rm
pa}(r)=\emptyset$. A vertex of $T$ of degree one is called a
\textit{leaf}. A vertex of $T$ that is not a leaf is called an
\textit{inner node}.

Let $T$ denote an undirected tree with $n$ leaves and let $T^{r}=(V,E)$
denote $T$ rooted in $r\in V$. A Markov process on a rooted tree $T^r$
is a sequence $\{Y_v:\,v\in V\}$ of random variables such that for each
$\alpha=(\alpha_v)_{v\in V}\in\{0,1\}^{V}$ its joint distribution
satisfies
%
\begin{equation}\label{eq:p_albar}
p_\alpha(\theta)=\theta^{(r)}_{\alpha_{r}}\prod_{v\in V\setminus r}
\theta^{(v)}_{\alpha_{v}|\alpha_{{\rm pa}(v)}},
\end{equation}
where $\theta^{(r)}_{\alpha_{r}}=\mathbb{P}(Y_{r}=\alpha_{r})$ and
$\theta^{(v)}_{\alpha_{v}|\alpha_{{\rm
pa}(v)}}=\mathbb{P}(Y_v={\alpha}_v|Y_{{\rm pa}(v)}=\alpha_{{\rm
pa}(v)})$. Since $\theta^{(r)}_{0}+\theta^{(r)}_{1}=1$ and
$\theta^{(v)}_{0|i}+\theta^{(v)}_{1|i}=1$ for all $v\in
V\setminus\{r\}$ and $i=0,1$ then the set of parameters consists of
exactly $2|E|+1$ free parameters: we have two parameters:
$\theta^{(v)}_{1|0}$, $\theta^{(v)}_{1|1}$ for each edge $(u,v)\in E$
and one parameter $\theta^{(r)}_{1}$ for the root. We denote the
parameter space by $\Theta_{T}=[0,1]^{2|E|+1}$ and the Markov process
on $T^{r}$ by $\widetilde{\mathcal{M}}_{T}$.
\begin{rem}\label{rem:rootings}
The reason to omit the root $r$ in the notation is that this model does
not depend on the rooting and is equivalent to the undirected graphical
model given by global Markov properties on $T$. To prove this note that
$T^{r}$ is a perfect directed graph and hence by \cite[Proposition
3.28]{lauritzen:96} parametrization in (\ref{eq:p_albar}) is
equivalent to factorization with respect to $T$. Since $T$ is
decomposable, by \cite[Proposition 3.19]{lauritzen:96}, this
factorization is equivalent to the global Markov properties.
\end{rem}

Let $\Delta_{2^n-1}=\{p\in\mathbb{R}^{2^n}: \sum_\beta p_\beta=1,
p_\beta\geq0\}$ with indices $\beta$ ranging over $\{0,1\}^n$ be the
probability simplex of all possible distributions of $X=(X_1,\ldots,
X_n)$ represented by the leaves of $T$. We assume now that all the
inner nodes represent hidden variables. Equation (\ref{eq:p_albar})
induces a polynomial map $f_T:\Theta_{T}\rightarrow\Delta_{2^n-1}$
obtained by marginalization over all the inner nodes of $T$
%
\begin{equation}\label{eq:p_albar2}
p_{\beta}(\theta)=\sum_\mathcal{H} \theta^{(r)}_{\alpha_{r}}\prod_{v\in
V\setminus r}\theta^{(v)}_{\alpha_{v}|\alpha_{{\rm pa}(v)}},
\end{equation}
where $\mathcal{H}$ is the set of all $\alpha\in\{0,1\}^{V}$ such that
the restriction to the leaves of $T$ is equal to $\beta$. We let
$\mathcal{M}_{T}=f_{T}(\Theta_{T})$ denote the \textit{general Markov
model} over the set of observable random variables (c.f. \cite[Section
8.3]{semple2003pol}).

A \textit{semialgebraic set} in $\mathbb{R}^{d}$ is a finite union of
sets given by a finite number of polynomial equations and inequalities.
Since $\Theta_{T}$ is a {semialgebraic set} and $f_{T}$ is a polynomial
map then by \cite[Proposition 2.2.7]{bochnak1998rag} $\mathcal{M}_{T}$
is a semialgebraic set as well. Moreover, if $f$ is a polynomial
isomorphism from $\Delta_{2^{n}-1}$ to another space then
$f(\mathcal{M}_{T})$ is also a semialgebraic set. The semialgebraic
description of $f(\mathcal{M}_{T})$ in $f(\Delta_{2^{n}-1})$ gives the
semialgebraic description of $\mathcal{M}_{T}$.

The idea behind tree cumulants was to define a polynomial isomorphism
from $\Delta_{2^{n}-1}$ to the space of new coordinates
$\mathcal{K}_{T}$. We defined a partially ordered set (poset) of all
the partitions of the set of leaves induced by removing edges of the
given tree $T$. Then tree cumulants are given as a function of
probabilities induced by a M\"{o}bius function on the poset. The
details of this change of coordinates are given in Appendix
\ref{app:change} and are illustrated below.

The tree cumulants are given by $2^{n}-1$ coordinates: $n$ means
${\lambda}_{i}=\mathbb{E} X_{i}$ for all $i=1,\ldots,n$ and a set of
real-valued parameters $\{\kappa_{I}:\,I\subseteq[n]\mbox{ where }
|I|\geq2\}$. Each formula for $\kappa_{I}$ is expressed as a function
of the higher order central moments of the observed variables. These
formulae are given explicitly in equation (\ref{eq:kappa-in-rho}) of
Appendix \ref{app:change}. Since the change of coordinates is a
polynomial isomorphism then, by \cite[Proposition
2.2.7]{bochnak1998rag}, the image of $\mathcal{M}_{T}$ in the space of
tree cumulants, denoted by $\mathcal{M}_{T}^{\kappa}$, is a
semialgebraic set. In this paper we provide the full semialgebraic
description of $\mathcal{M}_{T}^{\kappa}$, that is the complete set of
polynomial equations and inequalities involving the tree cumulants
which describes $\mathcal{M}_{T}^{\kappa}$ as the subset of
$\mathcal{K}_{T}$, for subsequent use in a statistical analysis of the
model class.

\begin{exmp}\label{ex:quartet1}
Consider the quartet tree model, i.e. the general Markov model given by
the graph in Figure \ref{fig:quartet}.
%
\begin{figure}
\includegraphics{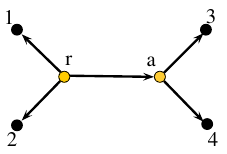}
\caption{A quartet tree}\label{fig:quartet}
\end{figure}
The tree cumulants are given by $15$ coordinates: ${\lambda}_{i}$ for
$i=1,2,3,4$ and $\kappa_{I}$ for $I\subseteq[4]$ such that $|I|\geq
2$. Denoting $U_{i}=X_{i}-\mathbb{E} X_{i}$ we have
$\kappa_{ij}=\mathbb{E} U_{i}U_{j}={\rm Cov}(X_{i},X_{j})$ for $1\leq
i<j\leq4$ and
\[
\kappa_{ijk}=\mathbb{E} \left(U_{i}U_{j}U_{k}\right)
\]
for all $1\leq i<j<k\leq4$ which we note is a third order central
moment. However, in general tree cumulants of higher order cannot be
equated with their corresponding central moments but only expressed as
functions of them. These functions are obtained by performing an
appropriate M\"{o}bius inversion. Thus for example from equation
(\ref{eq:kappa-in-rho}) in Appendix \ref{app:change} we have that
\[
\kappa_{1234}=\mathbb{E} \left(U_{1}U_{2}U_{3}U_{4}\right)-\mathbb{E}
\left(U_{1}U_{2}\right)\mathbb{E} \left(U_{3}U_{4}\right).
\]
Note that since the observed higher order central moments can be
expressed as functions of probabilities, tree cumulants can also be
expressed as functions of these probabilities.
\end{exmp}

Let $X_{\hat{i}}=(X_{1},X_{2},X_{3},X_{4})\setminus\{X_{i}\}$ for
$i=1,2,3,4$. From \cite[Proposition 4.3]{pwz-2010-cumulants} it follows
in particular that, like for the joint cumulant, $\kappa_{1234}=0$
whenever $X_{i}\indep X_{\hat{i}}$ for any $i=1,2,3,4$ or
$(X_{1},X_{2})\indep(X_{3},X_{4})$. However, in general,
$\kappa_{1234}\neq0$ for example if $(X_{1},X_{3})\indep
(X_{2},X_{4})$ and hence tree cumulants differ from classical
cumulants. Vanishing of the tree cumulants corresponds to an edge being
missing in the particular defining tree. This generalizes for other
trees and gives a heuristic explanation for the nice product-like
parametrization presented in Proposition \ref{prop:monomial} below. We
explain this now formally.

Let $T^{r}=(V,E)$ and let $\Omega_{T}$ denote the set of parameters
with coordinates given by $\bar{\mu}_{v}$ for $v\in V$ and $\eta_{u,v}$
for $(u,v)\in E$. Define a reparametrization map $f_{\theta\omega}:
\Theta_{T}\rightarrow\Omega_{T}$ as follows:
%
\begin{equation}\label{eq:uij}
\begin{array}{ll}
\eta_{u,v}=\theta^{(v)}_{1|1}-\theta^{(v)}_{1|0} & \mbox{for every
$(u,v)\in E$ and}\\
\bar{\mu}_v=1-2\lambda_v & \mbox{for each } v\in V,
\end{array}
\end{equation}
where $\lambda_{v}=\mathbb{E} Y_{v}$ is a polynomial in the original
parameters $\theta$. To see this let $r,v_{1},\ldots, v_{k},v$ be a
directed path in $T$. Then
%
\begin{equation}\label{eq:means-from-prob}
\lambda_{v}=\mathbb{P}(Y_{v}=1)=\sum_{\alpha\in\{0,1\}^{k+1}} \theta
^{(v)}_{1|\alpha_{k}}\theta^{(v_{k})}_{\alpha_{k}|\alpha_{k-1}}\cdots
\theta^{(r)}_{\alpha_{r}}.
\end{equation}
It can be easily checked that if ${\rm Var}(Y_{u})>0$ then
$\eta_{u,v}={\rm Cov}(Y_{u},Y_{v})/{\rm Var}(Y_{u})$. Hence
$\eta_{u,v}$ is just the regression coefficient of $Y_{v}$ with respect
to $Y_{u}$.

The parameter space $\Omega_{T}$ is given by the following constraints:
%
\begin{equation}\label{eq:constraints}
\begin{array}{l}
-1\leq\bar{\mu}_{r}\leq1, \qquad\mbox{and for each } (u,v)\in E\\
-(1+\bar{\mu}_{v}) \leq(1-\bar{\mu}_{u})\eta_{u,v} \leq(1-\bar{\mu
}_{v})\\
-(1-\bar{\mu}_{v}) \leq(1+\bar{\mu}_{u})\eta_{u,v} \leq(1+\bar{\mu}_{v}).
\end{array}
\end{equation}

In Appendix \ref{app:change} we show that there is a polynomial
isomorphism between $\Delta_{2^{n}-1}$ and the space of tree cumulants
$\mathcal{K}_{T}$ giving the following diagram, where the dashed arrow
denotes the induced parametrization.
%
\begin{equation}\label{eq:diagram}
\xymatrixcolsep{4pc}\xymatrix{
\Theta_{T} \ar@<1ex>[d]^{f_{\theta\omega}} \ar[r]^{f_{T}} &\Delta
_{2^{n}-1}\ar@<1ex>[d]^{f_{p\kappa}}\\
\Omega_{T}\ar@<1ex>[u]^{f_{\omega\theta}} \ar@{-->}[r]^{\psi_{T}}
&\mathcal{K}_{T}\ar@<1ex>[u]^{f_{\kappa p}}}
\end{equation}
One motivation behind this change of coordinates is that the induced
parametrization $\psi_{T}:\Omega_{T}\rightarrow\mathcal{K}_{T}$ has a
particularly elegant form.

\begin{prop}[\cite{pwz-2010-identifiability}, Proposition 4.1]
\label{prop:monomial} Let $T$ be an undirected tree with $n$ leaves.
Assume that $T$ is trivalent which here means that all of its inner
nodes have degree at most three. Let $T^{r}=(V,E)$ be $T$ rooted in
$r\in V$. Then $\mathcal{M}_T^\kappa$ is parametrized by the map
$\psi_{T}:\Omega_{T}\rightarrow\mathcal{K}_{T}$ given as
$\lambda_{i}=\frac{1}{2}(1-\bar{\mu}_{i})$ for $i=1,\ldots, n$ and
%
\begin{equation}\label{eq:kappa_def_general}
\kappa_{I}=\frac{1}{4}\left(1-\bar{\mu}_{r(I)}^{2}\right) \prod_{v\in
{\rm int}(V(I))} \bar{\mu}_{v}^{\deg(v)-2}\prod_{(u,v)\in E(I)} \eta
_{u,v}\quad\mbox{ for } I\subseteq[n], |I|\geq2
\end{equation}
where the degree is taken in $T(I)=(V(I),E(I))$; ${\rm int}(V(I))$
denotes the set of inner nodes of $T(I)$ and $r(I)$ denotes the root of
$T^{r}(I)$.
\end{prop}

Proposition \ref{prop:monomial} has been formulated for trivalent
trees. However, it can be easily extended to the general case as
explained in \cite[Section 4]{pwz-2010-identifiability}.

This result enabled us to completely understand identifiability of tree
models extending results in \cite{chang1996frm}. In particular \cite
[Theorem 5.4]{pwz-2010-identifiability} identifies the cases when the
model is identified up to label switching. This condition is rather
technical and here we usually would recommend the use of the sufficient
condition that all the covariances between the leaves are nonzero.
Further results focus on the geometry of the unidentified space in the
case when the identifiability fails. More importantly, \cite[Corollary
5.5]{pwz-2010-identifiability} gives us formulae for parameters given a
probability distribution in the case when identifiability holds. This
result gives us a closed-form formulae for MLEs in certain special
cases (see Corollary \ref{cor:MLEs}).

To illustrate our technique we next obtain the full semialgebraic
description of the tripod tree model. This result is not new (see e.g.
\cite{auvray2006sad,settimi1998gbg} and a special case given by
\cite[Theorem 3.1]{pearl1986fusion}). However, this allows us not only
to unify notation but also to introduce the strategy we use to prove
the general case. We begin with a definition.
\begin{defn}\label{def:hyperdet}
Let $A$ be a $2\times2\times2$ table. The hyperdeterminant of $A$ as
defined by Gelfand, Kapranov, Zelevinsky \cite[Chapter
14]{gelfand1994dra} is given by
\begin{eqnarray*}
{\rm Det} \,
A&=&(a_{000}^{2}a_{111}^{2}+a_{001}^{2}a_{110}^{2}+a_{010}^{2}a_{101}^{2}+a_{011}^{2}a_{100}^{2})\\
&-&2(a_{000}a_{001}a_{110}a_{111}+a_{000}a_{010}a_{101}a_{111}+a_{000}a_{011}a_{100}a_{111}\\
&+&a_{001}a_{010}a_{101}a_{110}+a_{001}a_{011}a_{110}a_{100}+a_{010}a_{011}a_{101}a_{100})\\
&+&4(a_{000}a_{011}a_{101}a_{110}+a_{001}a_{010}a_{100}a_{111}).
\end{eqnarray*}
\end{defn}
If $\sum a_{ijk}=1$ then treating all entries formally as joint cell
probabilities (without positivity constraints) we can simplify this
formula using the change of coordinates to central moments. The
reparametrizations in Appendix A are well defined for this extended
space of probabilities and we have that
%
\begin{equation}\label{eq:hyperdet}
{\rm Det} \,A=\mu_{123}^{2}+4\mu_{12}\mu_{13}\mu_{23},
\end{equation}
which can be verified by direct computations.

From the construction of tree cumulants (c.f. Appendix \ref
{app:change}) it follows that \mbox{$\kappa_{I}=\mu_{I}$} for all
$I\subseteq[n]$ such that $2\leq|I|\leq3$. Henceforth, for clarity,
these lower order tree cumulants will be written as their more familiar
corresponding central moments.
\begin{prop}[The semialgebraic description of the tripod model]\label
{lem:semi_tripod}
Let $\mathcal{M}_{3}$ be the general Markov model on a tripod tree $T$
rooted in any node of $T$. Let $P$ be a $2\times2\times2$ probability
table for three binary random variables $(X_{1}, X_{2}, X_{3})$ with
central moments $\mu_{12}, \mu_{13},\mu_{23}$, $\mu_{123}$ (equivalent
to the corresponding tree cumulants) and means $\lambda_{i}$, for
$i=1,2,3$. Then $P\in\mathcal{M}_{3}$ if and only if one of the
following two cases occurs:
\begin{itemize}
\item[(i)] $\mu_{123}=0$ and at least two of the three covariances $\mu
_{12},\mu_{13},\mu_{23}$ vanish.
\item[(ii)] $\mu_{12}\mu_{13}\mu_{23}> 0$ and
%
\begin{equation}\label{eq:series-ineq-tripod2}
\begin{array}{l}
|\mu_{jk}|\sqrt{{\rm Det}\,P}-\mu_{123}\mu_{jk}\leq(1-\bar{\mu}_{i})\mu
_{jk}^{2},\\
|\mu_{jk}|\sqrt{{\rm Det}\,P}+\mu_{123}\mu_{jk}\leq(1+ \bar{\mu
}_{i})\mu_{jk}^{2}
\end{array}
\end{equation}
for all $i=1,2,3$ where by $j,k$ we denote elements of $\{1,2,3\}
\setminus i$.
\end{itemize}
\end{prop}
\begin{proof}[Sketch of the proof]
The proof is given in Appendix \ref{sec:proofs}. Here, for convenience,
we give its outline. Denote by $\mathcal{M}\subseteq\Delta_{7}$ the
family of distributions described by (i) and (ii). We need to show that
$\mathcal{M}_{3}=\mathcal{M}$. To show that $\mathcal{M}_{3}\subseteq
\mathcal{M}$ we use the parametrization in Proposition \ref
{prop:monomial} to prove that either (i) holds or it does not, and
then, inequalities in (ii) are equivalent to (\ref{eq:constraints}). To
show the opposite inclusion we propose formulae for the parameters in
terms of the observed distribution given by \cite[Corollary
5.5]{pwz-2010-identifiability}, and show that this formulae agree with
the parametrization in Proposition \ref{prop:monomial} up to the sign.
The inequality $\mu_{12}\mu_{13}\mu_{23}>0$ assures that there is a
choice of signs for the parameters such that the parametrization holds exactly.
\end{proof}

All the points satisfying (i) correspond to submodels of $\mathcal{M}$
where some of the observed variables are independent of each other.

\section{Inferential issues related to the semialgebraic
description}\label{sec:inference}

There are at least three reasons why the implicit inequality
constraints of this model class can have a critical
impact on a statistical analysis of this model class. First, used in
conjunction with other geometric techniques these inequalities help us
determine, whether or not
the likelihood associated with a given tree model has multiple local
maxima. Second, it gives us the basis for developing simple model
diagnostics which complement those associated with implicit algebraic
constraints. Finally, awareness of whether these constraints are active
for given data set enables us to identify when standard numerical
methods might fail both for estimation and model selection across
different candidate trees. We consider and illustrate all these issues below.

Proposition \ref{lem:semi_tripod} and Theorem \ref{th:parameters0} give
explicit descriptions of tree models as subsets of the probability
simplex and hence also as submodels of the multinomial model. The
literature on constrained multinomial models (see \cite
{davis2009analysis} for a review) gives many examples of what may go
wrong in this case. If the multiway marginal table of observed random
variables is sampled at random then its likelihood will be given as the
multinomial likelihood constrained to the model. The unconstrained
multinomial likelihood is of course a very well-behaved function. In
particular it is log-concave and its unique maximum is given by the
sample proportions $\hat{p}$ as long as all the entries of $\hat{p}$
are nonzero. However, after constraining to the model this function may
become much more complicated.\looseness=1

We know that unidentifiability of parameters causes estimation problems
associated for example with multiple local maxima of the likelihood and
the posterior density. However, because the constraints on the model do
not define a convex region, the constrained likelihood will not
necessarily have a unique maximum (see Figure \ref{fig:cstrlike}). So
even if we use ways of cleverly accounting for the aliasing caused by
unidentifiability we can still be left with other multiple local
solutions induced by the violations of the constraints. This, in turn,
can make estimation schemes unstable. The discussion below complements
results presented in \cite{chor2000mml}.

If the unconstrained multinomial maximum likelihood estimator given by
the sample proportions satisfies the equation but does not satisfy some
of the inequalities then the MLE of the given tree model will \textit
{always} lie on the boundary of the parameter space $\Theta_{T}$. Of
course, if all the inequalities hold but some of the equalities do not
then, in principle, it is not such a serious problem as the estimates
will typically lie in the interior of the parameter space. However, if
there are even the smallest perturbations of the model class we are
likely to be drawn outside the feasible region. This is a phenomenon
observed in many applied analyzes of these models (see, e.g. \cite
{chor2000mml}). This occurs even in the simple tripod tree above where
the feasible region accounts for only 8\% of $\Delta_{7}$. Of course
simply sampling from the tree model itself will not identify this
potential difficulty since such samples will automatically not violate
the constraints in any significant way. But if the tree only
approximately holds then we begin to encounter certain difficulties.

\begin{figure}[t!]
\includegraphics[scale=0.3]{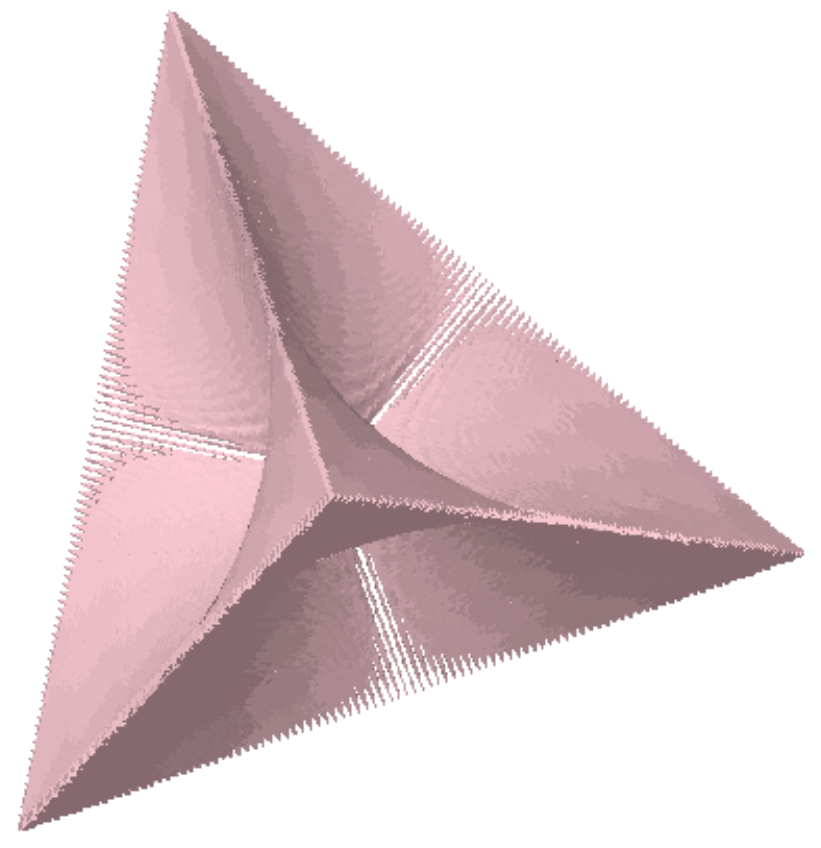}\hspace{.1cm}\includegraphics
[scale=0.3]{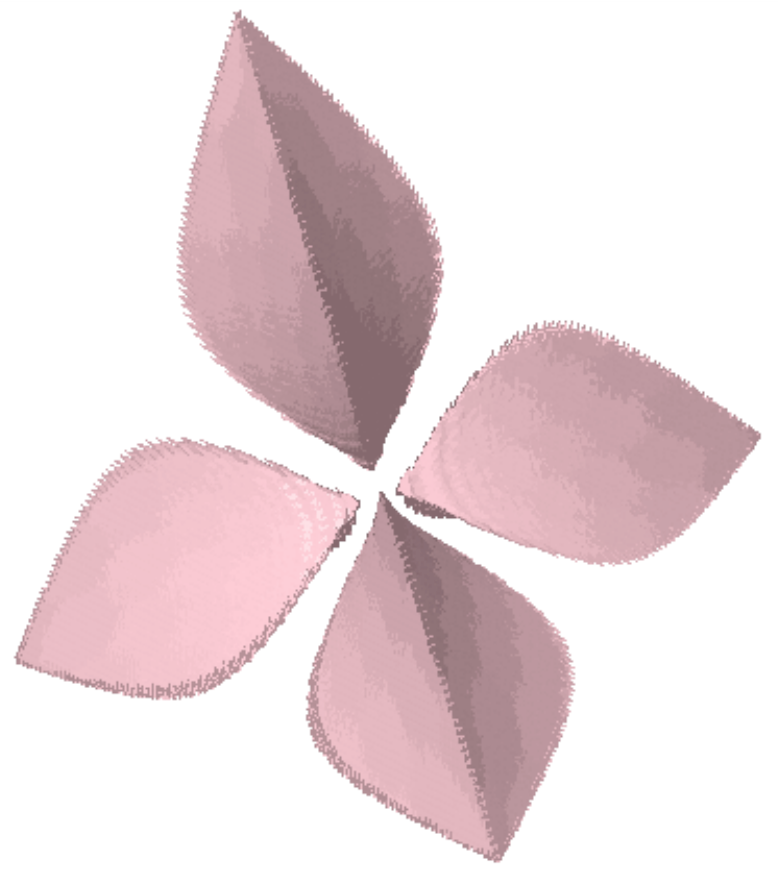}\hspace{.1cm}
\includegraphics[scale=0.3]{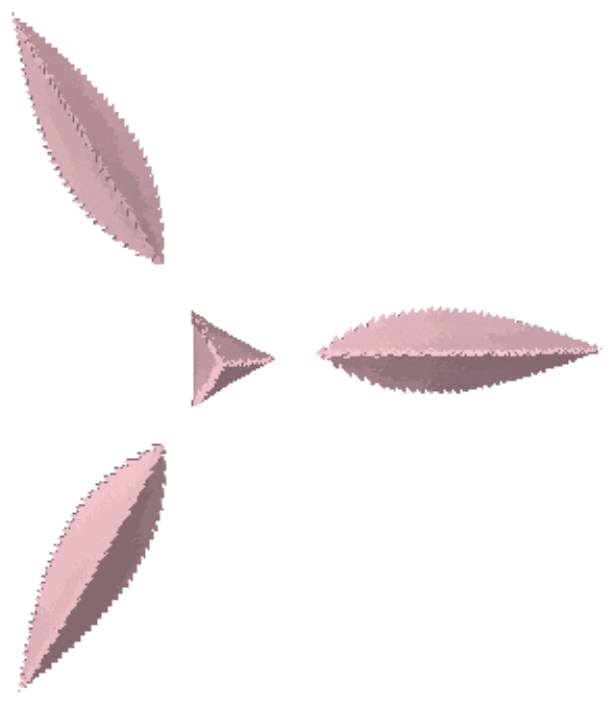}
\caption{The space of all possible covariances $\mu_{12},\mu_{13},\mu
_{23}$ for the tripod tree model in the case when $\lambda_{1}=\lambda
_{2}=\lambda_{3}=\frac{1}{2}$ and $\mu_{123}$ is equal to $0$, $0.005$
and $0.02$ (from left to right).}\label{fig:tripod-constr}
\end{figure}

\begin{figure}[h!]
\includegraphics[scale=0.5]{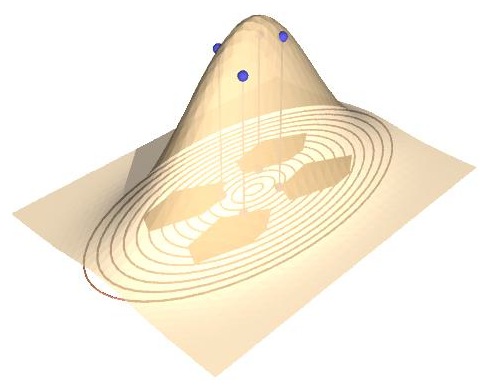}
\caption{The multinomial likelihood and a submodel of the saturated
model given by four disjoint regions. The four local maxima are
obtained on boundaries of these regions.}\label{fig:cstrlike}
\end{figure}

Since the tripod tree model $\mathcal{M}_{3}$ is of full dimension
there are no non-trivial phylogenetic invariants and so the feasible
regions of the model class are purely associated with inequality
constraints and so particularly straightforward. In Figure \ref
{fig:tripod-constr} we depict these constraints as they apply to the
second order moments of the three observed variables given
some typical values of the other coordinates. For example there are
four components corresponding to four possible choices of signs for
covariances satisfying $\mu_{12}\mu_{13}\mu_{23}\geq0$.


We can now give an explicit illustration of the type of multimodality
that can be induced in this context. The likelihood function $\ell:
\Theta_{T}\rightarrow\mathbb{R}$ for the tripod tree model can be also
treated as a function on $\Delta_{7}$ by $\ell(\theta)=\ell(p(\theta))$
in which case it will be denoted by $\ell(p)$. Since we understand the
parametrization $p:\,\Theta_{T}\rightarrow\Delta_{7}$ of $\mathcal
{M}_{3}$ then understanding $\ell(p)$ gives us automatically
understanding of $\ell(\theta)$. The advantage is that in this setting $
\ell(p)$ is just obtained as the multinomial likelihood function $\ell
(p)=\ell(p;x)=\prod p_{ijk}^{x_{ijk}}$ constrained to the model as
explained above. If $\hat{p}$ lies in the model class $\mathcal{M}_{3}$
then $\ell(p)$ has a unique maximum and the maxima of $\ell(\theta)$
can be obtained by mapping back $\hat{p}$ to the parameter space $\Theta
_{T}$ by using \cite[Equation (3)]{pwz-2010-identifiability}. This
result generalizes.
\begin{cor}\label{cor:MLEs}
Let $T=(V,E)$ be a phylogenetic tree with $n$ leaves and let $\mathcal
{M}_{T}$ be the corresponding tree model. If $\hat{p}\in\mathcal
{M}_{T}$ then \cite[Corollary 5.5]{pwz-2010-identifiability} gives the
formulae for the maximum likelihood estimators. In the case when the
number of MLEs is finite, there are always exactly $2^{|V|-n}$ MLEs
which are equivalent up to switching labels of the hidden variables.
\end{cor}

We have however argued that usually $\hat{p}\notin\mathcal{M}_{T}$. In
this case there is potentially more than one local maximum of the
constrained multinomial likelihood function. Let $\hat{p}$ the sample
proportions for some observed data on three binary random variables. We
have three possible scenarios:
\begin{itemize}
\item[(i)] $\hat{p}\in\mathcal{M}_{3}$ and then $\ell(p)$ is unimodal.
\item[(ii)] $\hat{p}\notin\mathcal{M}_{3}$ and $\ell(p)$ is multimodal
but there exists only one global maximum.
\item[(iii)] $\hat{p}\notin\mathcal{M}_{3}$ and $\ell(p)$ has multiple
global maxima.
\end{itemize}
The situation in (iii) raises an interesting question related to the
model identifiability. For every data point satisfying (iii) we are not
able to identify the parameters using the maximum likelihood estimation
even if we take into account the label switching problem.

Of course from the numerical point of view the situation in (ii) and
(iii) may describe equally bad scenarios since in both cases the
algorithms become unstable even for arbitrary large sample sizes. Thus
suppose that a sample of size $10000$ has been observed
%
\begin{equation}\label{eq:data}
\left[
\begin{array}{cc|cc}
x_{000} & x_{001} & x_{100} & x_{101}\\
x_{010} & x_{011} & x_{110} & x_{111}
\end{array}
\right]=\left[
\begin{array}{cc|cc}
2069 & 16 & 2242 & 331\\
2678& 863 & 442 & 1359
\end{array}
\right].
\end{equation}
By direct computations we check that all the constraint in Proposition
\ref{lem:semi_tripod} hold apart from $\mu_{12}\mu_{13}\mu_{23}\geq0$
and hence $\hat{p}$ does not lie in $\mathcal{M}_{3}$. The
corresponding parameters will lie on the boundary of the parameter
space. We performed the following simulation. We sampled uniformly from
$\Theta_{T}=[0,1]^{7}$ the starting parameters for the EM algorithm and
noted the results of the EM approximation. For $100$ iterations the
procedure found four different isolated maxima given in Table \ref{tab:em}.


\begin{table}
\caption{Results of the EM algorithm}\label{tab:em}
\vspace*{-15pt}
\[
\arraycolsep=11.5pt
\begin{array}{@{\ \ }c|ccccccc@{\ \ }}
\hline
& \theta_{1}^{(r)} & \theta_{1|0}^{(1)} & \theta_{1|1}^{(1)} & \theta
_{1|0}^{(2)} & \theta_{1|1}^{(2)}& \theta_{1|0}^{(3)}& \theta
_{1|1}^{(3)}\\
\hline
1&0.4658 & 0.3371 & 0.5524 & 1.0000 & 0.0000 & 0.4159& 0.0745\\
2&0.5342 & 0.5524 & 0.3371 & 0.0000 & 1.0000& 0.0745& 0.4159\\
3&0.4771 & 0.0000 & 0.9167 & 0.6369 & 0.4216 & 0.1468 & 0.3775\\
4&0.5229 & 0.9167 & 0.0000 & 0.4216 & 0.6369 & 0.3775 & 0.1468\\
\hline
\end{array}
\]
\vspace*{-12pt}
\end{table}

Up to label switching on the inner node these are two distinct
maximizers of the log-likelihood function $\ell(\theta)$ corresponding
to rows $1,3$. The value of the log-likelihood function, computed as
$\sum_{ijk}x_{ijk}\log p_{ijk}$, is equal to $-18387$ and $-18917$
respectively. Both points correspond to somewhat degenerate tripod tree
models where one of the observed variables is functionally related to
the hidden variable. For example the first point lies on the submodel
given by $X_{1}\indep X_{3}|X_{2}$. We performed a similar analysis for
other data points for which only $\mu_{12}\mu_{13}\mu_{23}\geq0$ fails
and three different EM maximizers were often found. In every case the
maximizers corresponded to degenerate submodels. In conjunction with
\cite[Theorem 5.4]{pwz-2010-identifiability} we also have data for
which the likelihood function $\ell(\theta)$ is maximized over an
infinite number of points. This for example holds for any data such
that the constrained multinomial likelihood is maximized over a point
such that $p_{0ij}=\lambda p_{1ij}$ for some $\lambda$ and each
$i,j=0,1$. In this case ${\mu}_{12}={\mu}_{13}=0$ and the MLEs form a
set of a positive dimension by \cite[Theorem~5.4]{pwz-2010-identifiability}.

We note that the whole discussion above remains valid for more general
tree models. The conditional independence properties of tree models
imply that, since any three leaves are separated by an inner node, the
corresponding marginal distributions form a tripod tree model.
Demanding that tripod tree constraints must be satisfied by \textit
{all} triples of observed random variables cuts out all but a small
proportion of the probability simplex. Furthermore by Theorem \ref
{th:parameters0} we know that, in addition, many other constraints
involving higher order moments will also apply. Therefore, the types of
issues we illustrated above become increasingly critical for inference
on trees, which in practical applications are of a much higher
dimension. Thus real-world data will typically satisfy all the
constraints defining the model very rarely. This, in turn, tends to
result in multimodality of the likelihood function and MLEs lying on
the boundary of the parameter space.

By acknowledging the existence of the inequality constraints we have
already demonstrated how graphical methods can be used to identify why
and where the fitted tree model might be flawed. Most naively, when
samples are very large we could calculate the sample moments and notice
which inequality constraints are active on the data set presented. When
these lie outside these regions then we have strong information that
the fitted tree model is inappropriate and we can expect there to be
problems with both estimation - as illustrated above - and model
selection. Slightly more sophisticatedly we could also compare the
model MLE: constrained as it is by these inequalities, with the MLE in
the saturated model. Likelihood ratio statistics can then be used to
measure the extent of the model inaccuracy. Of course this comparison
can be performed directly. However, then we lose the geometrical
insight as to exactly why and how the model is failing. This insight
will be helpful in guiding us in identifying alternative models that
might better explain the data. We note that the likelihood ratio
statistics for a constrained multinomial model against the saturated
model in general will not asymptotically have the $\chi^{2}$
distribution (see e.g. \cite{chernoff1954distribution}). If the
constrains are linear then the underlying distribution is called the
chi-bar squared distribution (see \cite{davis2009analysis}). The
situation is however much more complicated for tree models since here
the constraints define a union of non-convex bodies. In the end of
Section \ref{sec:metrics} we provide a short discussion on a
description of $\mathcal{M}_{T}$ in terms of convex sets.

Inequalities are also relevant for the model choice. Suppose that the
sufficient statistic does not satisfy some inequalities for each of the
models under analysis. Then asymptotic model selection techniques like
BIC can mislead. The effective parameter size will be miscounted
because at least some of the MLEs will lie of the boundary of the space
(see e.g. \cite{rusakov2006ams,pwz-2010-bic}). Model selection
based on Bayes factors will also tend to be unrobust. Since the
estimates lie on the boundary the marginal likelihood for each of the
models depends heavily on the tail behavior of the prior distribution
on that boundary. See \cite{smith2010robustness} and \cite
{smith2008isoseparation} for explanations of why this is so. For
example a standard choice of a prior distribution for conditional
distributions in tree models is the Dirichlet distribution. However,
for different choices of its prior parameters the Bayes factors
generated by the prior tails can be very different. Note that within
the Bayesian paradigm the sampling of the tripod tree is
straightforward once we recognize the constraint structure using a
simple importance sampler generating samples from $\Delta_{7}$ and
rejecting if they do not satisfy the defining inequalities. Of course
this is not the only way of specifying a prior density for selecting
between the saturated model and the tree model. However, our suggestion
is very simple to implement and its inferential consequences are more
transparent than more conventional methods using default priors within
the conventional probabilitistic parametrization, where the selection
can be highly dependent on the tails of priors.

\section{Explicit expression of implied inequality constraints}\label
{sec:metrics}

In this section we discuss the geometry of general tree models. First,
we use some links to tree metrics to provide a simple set of algebraic
constraints on the model space. Then, in Theorem \ref{th:parameters0},
we provide the complete semialgebraic description for this model class.

Let $T=(V,E)$ be a general undirected tree with $n$ leaves and $T^{r}$
the tree $T$ rooted in $r\in V$. Before stating the main theorem of the
paper we first show how to obtain an elegant set of necessary
constraints on $\mathcal{M}_{T}$. In this section we assume that $\bar
{\mu}_{r}^{2}\neq1$ and $\eta_{u,v}\neq0$ for all $(u,v)\in E$. By
\cite[Remark 4.3]{pwz-2010-identifiability}, this implies that $\bar{\mu
}_{v}^{2}\neq1$ for all $v\in V$. Since ${\rm Var}(Y_{u})=\frac
{1}{4}(1-\bar{\mu}_{u}^{2})$ the correlation between $Y_{u}$ and
$Y_{v}$ is defined as $\rho_{uv}=\frac{4\mu_{uv}}{\sqrt{(1-\bar{\mu
}_{u}^{2})(1-\bar{\mu}_{v}^{2})}}$. This gives
%
\begin{equation}\label{eq:rhouv}
\rho_{uv}=\eta_{u,v}\sqrt{\frac{1-\bar{\mu}_{u}^{2}}{1-\bar{\mu
}_{v}^{2}}}=\eta_{v,u}\sqrt{\frac{1-\bar{\mu}_{v}^{2}}{1-\bar{\mu}_{u}^{2}}}.
\end{equation}
\begin{lem}\label{lem:prodrho}
For any $i,j\in[n]$ let $E({ij})$ be the set of edges on the unique
path joining $i$ and $j$ in $T$. Then
%
\begin{equation}\label{eq:prod_rij}
\rho_{ij}=\prod_{(u,v)\in E({ij})}\rho_{uv}
\end{equation}
for each probability distribution in ${\mathcal{M}}_T^{\kappa}$ such
that all the correlations are well defined.
\end{lem}
\begin{proof}
By (\ref{eq:kappa_def_general}) applied to $T(ij)$ we have $\mu
_{ij}=\frac{1}{4}(1-\bar{\mu}_{r}^{2})\prod_{(u,v)\in E(ij)}\eta
_{u,v}$, where $r$ is the root of the path between $i$ and $j$ and hence
\[
\rho_{ij}=\sqrt{\frac{1-\bar{\mu}_{r}^{2}}{1-\bar{\mu}_{i}^{2}}}\sqrt
{\frac{1-\bar{\mu}_{r}^{2}}{1-\bar{\mu}_{j}^{2}}}\prod_{(u,v)\in
E(ij)}\eta_{u,v}.
\]
Now apply (\ref{eq:rhouv}) to each $\eta_{u,v}$ in the product above to
show (\ref{eq:prod_rij}).
\end{proof}

The above equation allows us to demonstrate an interesting
reformulation of our problem in term of tree metrics (c.f. \cite
[Section 7]{semple2003pol}) which we explain below (see also Cavender
\cite{cavender1997443}).
\begin{defn}\label{def:treemetric}
A function $\delta: \,[n]\times[n]\rightarrow\mathbb{R}$ is called a
\textit{tree metric} if there exists a tree $T=(V,E)$ with the set of
leaves given by $[n]$ and with a positive real-valued weighting $w:
E\rightarrow\mathbb{R}_{>0}$ such that for all $i,j\in[n]$
\[
\delta(i,j)=\left\{
\begin{array}{ll}
\sum_{e\in E(ij)} w(e),&\mbox{ if } i\neq j,\\
0, & \mbox{otherwise}.
\end{array}
\right.
\]
\end{defn}

Let now $d: V\times V\rightarrow\mathbb{R}$ be a map defined as
\[
d(k,l)=\left\{
\begin{array}{ll}
-\log(\rho_{kl}^{2}), & \mbox{for all } k,l\in V \mbox{ such that } \rho
_{kl}\neq0,\\
+\infty, & \mbox{otherwise}
\end{array}
\right.
\]
then $d(k,l)\geq0$ because $\rho_{kl}^{2}\leq1$ and $d(k,k)=0$ for
all $k\in V$ since $\rho_{kk}=1$. If $K\in\mathcal{M}_T^\kappa$ then by
(\ref{eq:prod_rij}) $\rho_{ij}^{2}=\prod_{e\in E(ij)}\rho_{e}^{2}$ and
we can define map $d_{(T;K)}: [n]\times[n]\rightarrow\mathbb{R}$
%
\begin{equation}\label{eq:metric}
-\log(\rho_{ij}^{2})=d_{(T;K)}(i,j)=\left\{
\begin{array}{ll}
\sum_{(u,v)\in E(ij)} d(u,v), & \mbox{if } i\neq j,\\
0, & \mbox{otherwise.}
\end{array}
\right.
\end{equation}
This map is a tree metric by Definition \ref{def:treemetric}. In our
case we have a point in the model space defining all the second order
correlations and $d_{(T;K)}(i,j)$ for $i,j\in[n]$. The question is:
What are the conditions for the ``distances'' between leaves so that
there exists a tree $T$ and edge lengths $d(u,v)$ for all $(u,v)\in E$
such that (\ref{eq:metric}) is satisfied? Or equivalently: What are the
conditions on the absolute values of the second order correlations in
order that $\rho_{ij}^{2}=\prod_{e\in E_{ij}}\rho_{e}^{2}$ (for some
edge correlations) is satisfied? We have the following theorem.
\begin{thm}[Tree-Metric Theorem, Buneman \cite{buneman1974nmp}]\label
{th:3metric}
A function $\delta:\, [n]\times[n]\rightarrow\mathbb{R}$ is a tree
metric on $[n]$ if and only if for every four (not necessarily
distinct) elements $i,j,k,l\in[n]$,
\[
\delta(i,j)+\delta(k,l)\leq\max\left\{\delta(i,k)+\delta(j,l),\delta
(i,l)+\delta(j,k) \right\}.
\]
Moreover, a tree metric defines the tree uniquely.
\end{thm}

This theorem gives us a set of explicit constraints on the
distributions in a tree model. Since $\delta(i,j)=\log(-\rho_{ij})$ the
constraints in Theorem \ref{th:3metric} translate in terms of
correlations to
\[
-\log(\rho_{ij}^{2}\rho_{kl}^{2})\leq-\min\{\log(\rho_{ik}^{2}\rho
_{jl}^{2}),\log(\rho_{il}^{2}\rho_{jk}^{2}) \} .
\]
Since $\log$ is a monotone function we obtain
%
\begin{equation}\label{eq:suffic-ineq1}
\min\left\{\frac{\rho_{ik}^{2}\rho_{jl}^{2}}{\rho_{ij}^{2}\rho
_{kl}^{2}},\frac{\rho_{il}^{2}\rho_{jk}^{2}}{\rho_{ij}^{2}\rho
_{kl}^{2}} \right\}=\min\left\{\frac{\mu_{ik}^{2}\mu_{jl}^{2}}{\mu
_{ij}^{2}\mu_{kl}^{2}},\frac{\mu_{il}^{2}\mu_{jk}^{2}}{\mu_{ij}^{2}\mu
_{kl}^{2}} \right\}\leq1
\end{equation}
for all not necessarily distinct leaves $i,j,k,l\in[n]$. Hence, using
the relation between correlations and tree metrics given in \cite
{cavender1997443} we managed to provide a set of simple semialgebraic
constraints on the model. Furthermore, later in Theorem \ref
{th:parameters0} we show that these constraints are not the only active
constraints on the model $\mathcal{M}_{T}$. Before we present this
theorem it is helpful to make some simple observations about the
relationship between correlations and probabilistic tree models.

Since $\rho_{uv}$ can have different signs we define a signed tree
metric as a tree metric with an additional sign assignment for each
edge of $T$.
\begin{lem}\label{lem:inequalities2}
Let $T$ be a tree with set of leaves $[n]$. Suppose that we have a map
$\sigma:[n]\times[n]\rightarrow\{-1,1\}$. Then there exists a map
$s_{0}:E\rightarrow\{-1,1\}$ such that for all $i,j\in[n]$
%
\begin{equation}\label{eq:sij}
\sigma(i,j)=\prod_{(u,v)\in E(ij)}s_{0}(u,v)
\end{equation}
if and only if for all triples $i,j,k\in[n]$ $\sigma(i,j)\sigma
(i,k)\sigma(j,k)=1$.
\end{lem}
The proof is given in Appendix \ref{sec:proofs}.

The following proposition gives a set of simple constraints on
probability distribution in tree models. This may be particularly
useful in practice since it involves only computing pairwise margins of
the data and it enables us to check if a data point may come from a
phylogenetic tree model.
\begin{prop}\label{prop:simple-ineq}
Let $P\in\Delta_{2^{n}-1}$ be a probability distribution. If $P\in
\mathcal{M}_{T}$ for some tree $T$ with $n$ leaves then
%
\begin{equation}\label{eq:suffic-ineq2}
0\leq\min\left\{\frac{\mu_{ik}\mu_{jl}}{\mu_{ij}\mu_{kl}},\frac{\mu
_{il}\mu_{jk}}{\mu_{ij}\mu_{kl}} \right\}\leq1
\end{equation}
for all (not necessarily distinct) $i,j,k,l\in[n]$ whenever $\mu
_{ij},\mu_{kl}\neq0$.
\end{prop}
\begin{proof}
Lemma \ref{lem:inequalities2} implies that for all $i,j,k\in[n]$
necessarily $\mu_{ij}\mu_{ik}\mu_{jk}\geq0$. This in particular
implies that $\frac{\mu_{ik}\mu_{jl}}{\mu_{ij}\mu_{kl}}\geq0$ for all
$i,j,k,l\in[n]$. By taking the square root in (\ref{eq:suffic-ineq1})
these constraints can be combined to give the inequalities in~(\ref{eq:suffic-ineq2}).
\end{proof}

In Theorem \ref{th:parameters0} we show that (\ref{eq:suffic-ineq2})
provides the complete set of inequality constraints on $\mathcal
{M}_{T}$ that involve only second order moments in their expression.
The fact that additional constraints involving higher order moments
exist is illustrated in the following simple example.
\begin{exmp}
Consider the tripod tree model in Proposition \ref{lem:semi_tripod}.
Let $K$ be a point in $\mathcal{K}_{T}$ given by ${\lambda}_{i}=0.15$
for $i=1,2,3$, $\mu_{ij}=0.0625$ (or equivalently $\rho_{ij}=0.49$) for
each $i<j$ and $\mu_{123}= 0.0526$. This point lies in the space of
tree cumulants $\mathcal{K}_{T}$ which can be checked by mapping back
the central moments to probabilities, since the resulting vector
$[p_{\alpha}]$ lies in $\Delta_{7}$.

Clearly $K$ satisfies all the tree metric constraints in (\ref
{eq:suffic-ineq2}). The equation (\ref{eq:prod_rij}) is satisfied with
$\rho_{hi}=0.7$ for each $i=1,2,3$. We now show that despite this
$K\notin\mathcal{M}_{T}^{\kappa}$. For if $K\in\mathcal{M}_{T}^{\kappa
}$ then we could find $\bar{\mu}_{h}$ and $\eta_{h,i}$ satisfying
constraints in (\ref{eq:constraints}) so that (\ref{eq:star}) held.
Using the formulae in \cite[Corollary 5.5]{pwz-2010-identifiability} it
is easy to compute that $\bar{\mu}_{h}=0.86$ and $\eta_{h,i}\approx
0.98$. However, $K$ is not in the model since these parameters do not
lie in $\Omega_{T}$. Indeed,
\[
(1+\bar{\mu}_{h})\eta_{h,i}\approx1.8228 > (1+\bar{\mu}_{i})=1.7
\]
and hence (\ref{eq:constraints}) is not satisfied.

The consequence of the fact that the parameters do not lie in $\Omega
_{T}$ is that this parametrization does not lead to a valid assignment
of conditional probabilities to the edges of the tree. For example with
the values given above we can calculate that the induced marginal
distribution for $(X_{i},H)$ would have to satisfy \mbox{$\mathbb
{P}(X_{i}=0,H=1)=-0.0043$} which is obviously not a consistent
assignment for a probability model. Thus, there must exist other
constraints involving observed higher order moments that need to hold
for a probability model to be valid. We note that for the tripod tree
these were given by Proposition \ref{lem:semi_tripod}.
\end{exmp}

%
%

The following theorem gives the complete set of constraints which have
to be satisfied by tree cumulants to lie in $\mathcal{M}_{T}$ in the
case when $T$ is a trivalent tree. Let $P\in\Delta_{2^{n}-1}$ be the
probability distribution of the vector $(X_{1},\ldots, X_{n})$ then for
any $i,j,k\in[n]$ let $P^{ijk}$ denote the $2\times2\times2$ table of
the marginal distribution of $(X_{i},X_{j},X_{k})$.
%
\begin{thm}\label{th:parameters0}
Let $T=(V,E)$ be a trivalent tree with $n$ leaves and $\mathcal
{M}_{T}\subseteq\Delta_{2^{n}-1}$ be the model defined as an image of
the parametrization in (\ref{eq:p_albar2}). Suppose $P$ is a joint
probability distribution on $n$ binary variables. Then $P\in\mathcal
{M}_{T}$ if and only if the following conditions hold:
\begin{description}
\item[(C1)] For each edge split $A|B$ (c.f. Definition \ref
{def:edge-part}) of the set of leaves of $T$ whenever we have four
nonempty subsets (not necessarily disjoint) $I_{1},I_{2}\subseteq A$,
$J_{1},J_{2}\subseteq B$ then
\[
\kappa_{I_{1}J_{1}}\kappa_{I_{2}J_{2}}-\kappa_{I_{1}J_{2}}\kappa_{I_{2}J_{1}}=0.
\]
\item[(C2)] For all $1\leq i<j<k\leq n$ the corresponding marginal
distribution $P^{ijk}$ lies in the tripod model.
\item[(C3)] for all $I\subseteq[n]$ if there exist $i,j\in I$ such that
$\mu_{ij}=0$ then $\kappa_{I}=0$\\[-.2cm]
\item[(C4)] for any $ i,j,k,l\in[n]$ such that there exists $e\in E$
inducing a split $A|B$ such that $i,j\in A$ and $k,l\in B$ we have
\begin{equation*}\label{eq:ineqII}
\begin{array}{l}
(2\mu_{ik}\mu_{jl})^{2}\leq(\sqrt{\mu_{jl}^{2}{\rm Det} \,P^{ijk}}\pm
\mu_{jl}\mu_{ijk})(\sqrt{{\rm Det} \,P^{ikl}}\mp\mu_{ikl}).
\end{array}
\end{equation*}
\end{description}
Moreover, if $\mu_{ij}\neq0$ for all $i,j\in[n]$ then the constraints
in Proposition \ref{prop:simple-ineq} are the only constraints
involving only second order moments.
\end{thm}
\begin{proof}[Sketch of the proof]
The proof is given in Appendix \ref{sec:proof}. Here, for convenience,
we give its outline. Denote by $\mathcal{M}\subseteq\Delta_{2^{n}-1}$
the family of distributions described by (C1)-(C4). We need to show
that $\mathcal{M}_{T}=\mathcal{M}$. To show that $\mathcal
{M}_{T}\subseteq\mathcal{M}$ we use the parametrization in Proposition
\ref{prop:monomial} to show that (C1) and (C3) always hold, and that
(C2) and (C4) are equivalent to (\ref{eq:constraints}). To show the
opposite inclusion we propose formulae for the parameters in terms of
the observed distribution given by \cite[Corollary
5.5]{pwz-2010-identifiability}, and show that this formulae agree with
the parametrization in Proposition \ref{prop:monomial} up to the sign.
The last part is technical since we need to show that (C1)-(C4) also
imply that there is a choice of signs for the parameters such that the
parametrization in Proposition \ref{prop:monomial} holds exactly.
\end{proof}


Theorem \ref{th:parameters0} has been formulated for trivalent trees.
However, any tree with degrees of some nodes higher than three can be
realized as a submodel of a trivalent tree model as explained in \cite
[Section 4]{pwz-2010-identifiability}. Also, including degree two nodes
does not change anything in the induced marginal distribution. This
result is well known (see e.g. \cite[Lemma 2.1]{pwz-2010-identifiability}).

A natural question arises for how large trees it is feasible to verify
the constraints defining the model. The equality constraints in (C1)
can be expressed directly in the raw probabilities and they are easy to
check even for relatively large trees. This, by \cite[Theorem
4]{allman2008pia},can be done using so called edge flattenings, which
is explained in more details in Appendix \ref{sec:invariants}. Checking
the other constraints requires only computing ${n\choose2}$
covariances between the observed variables and ${n\choose3}$ third
order central moments. In particular, in practice there is no need of
changing the coordinates from the raw probabilities to tree cumulants
which can be quite complicated even for relatively small trees.


Another important practical aspect is whether there exist some
efficient convex bounds for the model in the space of the raw
probabilities. The answer to this question is negative, which follows
from the fact that ${\rm conv}(\mathcal{M}_{T})=\Delta_{2^{n}-1}$. This
is easily seen from the fact that $\mathcal{M}_{{\rm ind}}\subseteq
\mathcal{M}_{T}$, where $\mathcal{M}_{{\rm ind}}$ denotes the model of
full independence $X_{1}\indep\ldots\indep X_{n}$, and that ${\rm
conv}(\mathcal{M}_{{\rm ind}})=\Delta_{2^{n}-1}$. To get some
informative convex bounds one possibility is to generalize the tripod
tree case. Here the model consists of four components depicted in
Figure \ref{fig:cstrlike} corresponding to different sign patterns of
the observed covariances. These components are equivalent up to
rotation and symmetry. Instead of taking the convex hull of the whole
model we suggest the analysis of the convex hull of each of the
components separately. This is also well motivated by the fact that in
phylogenetics it is usually assumed that $\eta_{u,v}>0$ for all
$(u,v)\in E$ which means restriction to one of the components with all
the observed covariances positive. We will not discuss this issue here
in more detail.

\section{Example: The quartet tree model}\label{sec:quartet}

We can check that the point $K\in\mathcal{K}_{T}$ provided in Table \ref
{tab:quartet1} satisfies all the constraints in Theorem \ref{th:parameters0}.
%
\begin{table}
\caption{Moments and tree cumulants for a probability assignment which
lies in $\mathcal{M}_{T}$, where $T$ is the quartet tree}\label{tab:quartet1}
{
\vspace*{-15pt}
\[
\arraycolsep=11.5pt
\begin{array}{@{\ \ }ccccc@{\ \ }}
\hline
\alpha& I & p_{\alpha} & \lambda_{I} & \kappa_{I}\\
\hline
0000 & \emptyset&\frac{163837}{1417176}& 1& 0\\[.1cm]
0001 & 4&\frac{100735}{1417176}& \frac{1}{2}& 0\\[.1cm]
0010 & 3&\frac{48167}{708588}& \frac{1}{2}& 0\\[.1cm]
0011 & 34&\frac{45955}{708588}& \frac{253}{972}& \frac{5}{486}\\[.1cm]
0100 & 2&\frac{85507}{1417176}& \frac{1}{2}& 0\\[.1cm]
0101 & 24&\frac{76007}{1417176}& \frac{251}{972}& \frac{2}{243}\\[.1cm]
0110 & 23&\frac{36559}{708588}& \frac{85}{324}& \frac{1}{81}\\[.1cm]
0111 & 234&\frac{35531}{708588}& \frac{2489}{17496}& \frac{4}{2187}\\[.1cm]
1000 & 1&\frac{41255}{708588}& \frac{1}{2}& 0\\[.1cm]
1001 & 14&\frac{37315}{708588}& \frac{253}{972}& \frac{5}{486}\\[.1cm]
1010 & 13&\frac{73199}{1417176}& \frac{43}{162}& \frac{5}{324}\\[.1cm]
1011 & 134&\frac{75355}{1417176}& \frac{1271}{8748}& \frac{5}{2187}\\[.1cm]
1100 & 12&\frac{43471}{708588}& \frac{829}{2916}& \frac{25}{729}\\[.1cm]
1101 & 124&\frac{44171}{708588}& \frac{8107}{52488}& \frac{20}{6561}\\[.1cm]
1110 & 123&\frac{97063}{1417176}& \frac{1405}{8748}& \frac{10}{2187}\\[.1cm]
1111 & 1234&\frac{130547}{1417176}& \frac{130547}{1417176}& \frac
{40}{59049}\\
\hline
\end{array}
\]
}
\vspace*{-6pt}
\end{table}
It is convenient to provide the numbers as rationals so that the
equalities can be checked exactly. To check (C1), note for example that
\[
\kappa_{13}\kappa_{24}-\kappa_{14}\kappa_{23}=\frac{5}{324}\cdot\frac
{2}{243}-\frac{5}{486}\cdot\frac{1}{81}=0,
\]
\vspace*{-3pt}
\[
\kappa_{123}\kappa_{134}-\kappa_{1234}\kappa_{13}=\frac{10}{2187}\cdot
\frac{5}{2187}-\frac{40}{59049}\cdot\frac{5}{324}=0.
\]
To check (C2) verify for example that ${\rm Det} P^{123}=\frac
{25}{531441}$ and
\[
\left((1\pm\bar{\mu}_{1})\mu_{23}\mp\mu_{123}\right)^{2}=\biggl\{\frac
{1369}{4782969},\frac{289}{4782969}\biggr\}
\]
\vspace*{-3pt}
\[
\left((1\pm\bar{\mu}_{2})\mu_{13}\mp\mu_{123}\right)^{2}=\biggl\{\frac
{30625}{76527504},\frac{9025}{76527504}\biggr\}
\]
\vspace*{-3pt}
\[
\left((1\pm\bar{\mu}_{3})\mu_{12}\mp\mu_{123}\right)^{2}=\biggl\{\frac
{7225}{4782969},\frac{4225}{4782969}\biggr\}
\]
and hence
\[
{\rm Det} P^{123} \leq\min\left\{\left((1\pm\bar{\mu}_{\sigma(i)})\mu
_{\sigma(j)\sigma(k)}\mp\mu_{ijk}\right)^{2}\right\}= \frac{289}{4782969}
\]
is satisfied.
\eject

From the point of view of the original motivation a different scenario
is of interest. Imagine that we have $K\in\mathcal{K}_{T}$ such that
all the equalities in (C1) are satisfied, i.e. all the phylogenetic
invariants hold. If one of the constraints in (C2)-(C5) does not hold
then $K\notin\mathcal{M}_{T}^{\kappa}$. This shows that the method of
phylogenetic invariants as commonly used can lead to spurious results.
For example consider sample proportions and the corresponding tree
cumulants as in Table~\ref{tab:SEMquartet2}.
%
\begin{table}
\caption{Moments and tree cumulants of the given probability assignment
which does not lie in $\mathcal{M}_{T}$}\label{tab:SEMquartet2}
{
\vspace*{-15pt}
\[
\arraycolsep=11.5pt
\begin{array}{@{\ \ }ccccc@{\ \ }}
\hline
\alpha& I & p_{\alpha} & \lambda_{I} & \kappa_{I}\\
\hline
0000 & \emptyset&\frac{163837}{1417176}& 1& 0\\[.1cm]
0001 & 4&\frac{83213}{1417176}& \frac{1}{2}& 0\\[.1cm]
0010 & 3&\frac{10999}{177147}& \frac{1}{2}& 0\\[.1cm]
0011 & 34&\frac{11519}{177147}& \frac{1009}{2916}& \frac{70}{729}\\[.1cm]
0100 & 2&\frac{105785}{1417176}& \frac{1}{2}& 0\\[.1cm]
0101 & 24&\frac{52489}{1417176}& \frac{97}{324}& \frac{4}{81}\\[.1cm]
0110 & 23&\frac{6875}{177147}& \frac{95}{324}& \frac{7}{162}\\[.1cm]
0111 & 234&\frac{8515}{177147}& \frac{4285}{17496}& \frac{56}{2187}\\[.1cm]
1000 & 1&\frac{13834}{177147}& \frac{1}{2}& 0\\[.1cm]
1001 & 14&\frac{7226}{177147}& \frac{283}{972}& \frac{10}{243}\\[.1cm]
1010 & 13&\frac{61777}{1417176}& \frac{139}{486}& \frac{35}{972}\\[.1cm]
1011 & 134&\frac{51137}{1417176}& \frac{6113}{26244}& \frac{140}{6561}\\[.1cm]
1100 & 12&\frac{13760}{177147}& \frac{293}{972}& \frac{25}{486}\\[.1cm]
1101 & 124&\frac{3088}{177147}& \frac{3749}{17496}& \frac{40}{2187}\\[.1cm]
1110 & 123&\frac{13445}{1417176}& \frac{1805}{8748}& \frac{35}{2187}\\[.1cm]
1111 & 1234&\frac{278965}{1417176}& \frac{278965}{1417176}& \frac
{560}{59049}\\
\hline
\end{array}
\]
}
\end{table}
It can be checked that for this point all the equations in (C1) are
satisfied. However, this point does not lie in the model space. Using
the formulae in \cite[Corollary 5.5]{pwz-2010-identifiability}, which
gives the inverse map for the parametrization, it is simple to confirm
that the point mapping to $K$ satisfies $\theta^{(4)}_{1|1}=\frac
{67}{54}>1$. This cannot therefore be a probability and so $\theta
\notin\Theta_{T}$.

%

\section{Discussion}

The new coordinate system proposed in \cite{pwz-2010-identifiability}
provides a better insight into the geometry of phylogenetic tree models
with binary observations. The product form of the parametrization is
useful and has already enabled us to obtain the full geometric
description of the model class.

Of course it is one thing formally being able to identify the
constraints in the model and quite another to use this understanding
for model selection and estimation in realistically large scale
problems. The results in this paper only formally allow us to determine
explicitly the extremely complex nature of the feasible solution space
of a given tree model and determine whether a proposed estimate is
feasible. So they simply represent the first stage in constructing
methodology which supports these insights with an inferential
technology that can address statistical issues in large tree. In
particular, there remains the much more challenging issue of designing
samplers that use our results explicitly to efficiently estimate and
explore the tree model space. We are currently investigating this issue
and hope to report such algorithms in a later paper.

One of the interesting implications of our results for phylogenetic
analysis is that it enables us to consider different, simpler model
classes containing the original one in such a way that the whole
evolutionary interpretation in terms of the tree topologies remains
valid. If we were interested only in the tree we could consider the
model defined only by a subsets of constraints in Theorem \ref
{th:parameters0} involving only covariances. The cost of this reduction
is that the conditional independencies induced by the original model no
longer hold, which, in turn, affects the interpretation of the model.
We note that this approach is in a similar spirit to that employed to
motivate the MAG model class introduced in \cite{spirtes97heuristic}.

\section*{Acknowledgments}

Diane Maclagan and John Rhodes contributed substantially to this paper.
We would also like to thank Bernd Sturmfels for a stimulating
discussion at the early stage of our work and Lior Pachter for pointing
out reference \cite{cavender1997443}.

\appendix
{

\section{Change of coordinates}\label{app:change}

In this section we index raw probabilities with subsets of $[n]$
instead of $\{0,1\}^{n}$. We identify $I\subseteq[n]$ with $\alpha\in
\{0,1\}^{n}$ such that $\alpha_{i}=1$ only if $i\in I$. We first change
our coordinates from the raw probabilities $p=[p_{I}]_{I\subseteq[n]}$
to the non-central moments $\mathbf{\lambda}=[\lambda_{I}]_{I\subseteq
[n]}$, where $\lambda_{I}=\mathbb{E} (\prod_{i\in I} X_{i})$. This is a
linear map $f_{p\lambda}:\mathbb{R}^{2^{n}}\rightarrow\mathbb
{R}^{2^{n}}$ with determinant equal to one, where the components
$\lambda_{I}$ of the vector $\mathbf{\lambda}=f_{p\lambda}(p)$ are
defined by
%
\begin{equation}\label{eq:lambda_in_p}
\lambda_I=\sum_{J\supseteq I} p_J\qquad\mbox{ for any } I\subseteq[n].
\end{equation}
In particular $\lambda_{\emptyset}=1$ for all probability distributions
and the image $f_{p\lambda}(\Delta_{2^{n}-1})$ is contained in the
hyperplane defined by $\lambda_{\emptyset}=1$. Moreover, from (\ref
{eq:lambda_in_p}), it follows that the $\lambda$'s are just marginal
probabilities.
The linearity of the expectation implies that the central moments can
be expressed in terms of non-central moments. Define $\mu_{I}=\mathbb
{E} (\prod_{i\in I} U_{i})$, where $U_{i}=X_{i}-\mathbb{E} X_{i}$. Then
%
\begin{equation}\label{eq:central}
\mu_I=\sum_{J\subseteq[n]}(-1)^{|J|} \lambda_{I\setminus J} \prod
_{i\in J} \lambda_{i} \quad\mbox{ for }I\subseteq[n].
\end{equation}
Using these equations we can transform coordinates from the non-central
moments $\lambda=[\lambda_{I}]$ to another set of variables given by
all the means $\lambda_{{1}},\ldots, \lambda_{{n}}$ and central moments
$[\mu_{I}]$ for $I\subseteq[n]$. The polynomial map $f_{\lambda\mu
}:\mathbb{R}^{2^{n}}\rightarrow\mathbb{R}^{{n}}\times\mathbb
{R}^{2^{n}}$ is an identity on the first $n$ coordinates corresponding
to the means $\lambda_{{1}},\ldots, \lambda_{{n}}$ and is defined on
the remaining coordinates using the equations (\ref{eq:central}). Let
$\mathcal{C}_{n}=(f_{\lambda\mu} \circ f_{p\lambda})(\Delta
_{2^{n}-1})$. This is contained in a subspace of $\mathbb{R}^{n}\times
\mathbb{R}^{2^{n}}$ given by
\[
\mu_{\emptyset}=1 \quad\mbox{ and }\quad\mu_{{1}}=\cdots=\mu_{{n}}=0.
\]
Since $f_{\lambda\mu}$ is invertible (see \cite[Appendix
A.1]{pwz-2010-identifiability}) it provides a change of coordinates
from the non-central moments to a coordinate system on $\mathcal
{C}_{n}$ given by $\lambda_{1},\ldots, \lambda_{n}$ together with $\mu
_{I}$ for all $I\subseteq[n]$ such that $|I|\geq2$. Note that the
Jacobian of $f_{\lambda\mu}\circ f_{p\lambda}: \Delta
_{2^{n}-1}\rightarrow\mathcal{C}_{n}$ is constant and equal to one.

The final change of coordinates requires some combinatorics.
\begin{defn}\label{def:edge-part}
Let $T=(V,E)$ be a tree with $n$ leaves. An \textit{edge split} is a
partition of $[n]$ into two non-empty sets induced by removing an edge
$e\in E$ and restricting $[n]$ to the connected components of the
resulting graph. By an \textit{edge partition} we mean any partition
$B_1|\cdots|B_k$ of the set of leaves of $T$ induced by removing a
subset of $E$. Each $B_i$ is called a \textit{block} of the partition.
\end{defn}

Let $\Pi_{T}$ denote the partially ordered set (poset) of all tree
partitions of the set of leaves. The ordering in this poset is induced
from the ordering in the lattice $\Pi_{n}$ of all partitions of $[n]$
(see \cite[Example 3.1.1.d]{stanley2006enumerative}). Thus for $\pi
=B_{1}|\cdots|B_{r}$ and $\nu=B'_{1}|\cdots|B_{s}'$ we have $\pi\leq\nu
$ if every block of $\pi$ is contained in one of the blocks of $\nu$.
The poset $\Pi_{T}$ has a unique minimal element $1|2|\cdots|n$ induced
by removing all edges in $E$ and the maximal one with no edges removed
which is equal to a single block $[n]$. The maximal element is denoted
by $\hat{1}$ and the minimal one is denoted by $\hat{0}$.

For any poset $\Pi$ a \textit{M\"{o}bius function} $\mathfrak{m}_\Pi:\Pi
\times\Pi\rightarrow\mathbb{R}$ can be defined in such a way that
$\mathfrak{m}_\Pi(\pi,\pi)=1$ for every $\pi\in\Pi$,
$\mathfrak{m}_\Pi(\nu,\pi)=-\sum_{\nu\leq\delta<\pi} \mathfrak{m}_\Pi
(\nu,\delta)$ for $\nu<\delta$ in $\Pi$ and is zero otherwise (c.f.
\cite[Section 3.7]{stanley2006enumerative}). Let $T(W)$, for $W\subset
V$, denote the minimal subtree of $T$ containing $W$ in its set of
vertices. Then $\Pi_{T(W)}$ is the poset of all multisplits of the set
of leaves of $T(W)$ induced by edges of $T(W)$. The M\"{o}bius function
on $\Pi_{T(W)}$ will be denoted by $\mathfrak{m}_W$ and the M\"{o}bius
function on $\Pi_{T}$ will be denoted by $\mathfrak{m}$. Let $\hat
{0}_W$ and $\hat{1}_W$ denote the minimal and the maximal element of
$\Pi_{T(W)}$ respectively.

Consider a map $f_{\mu\kappa}: \mathbb{R}^{n}\times\mathbb
{R}^{2^{n}}\rightarrow\mathbb{R}^{n}\times\mathbb{R}^{2^{n}}$ where
the coordinates in the domain are denoted by $\lambda_{1},\ldots,
\lambda_{n}$ and $\mu_{I}$ for $I\subseteq[n]$ and let the coordinates
of the image space be denoted by $\lambda_{1},\ldots, \lambda_{n}$ and
$\kappa_{I}$ for $I\subseteq[n]$. The map is defined as the identity
on the first $n$ coordinates corresponding to $\lambda_{1},\ldots,
\lambda_{n}$ and
%
\begin{equation}\label{eq:kappa-in-rho}
\kappa_{I}=\sum_{\pi\in\Pi_{{T}(I)}}\mathfrak{m}_I(\pi,\hat{1}_{I})
\prod_{B\in\pi}\mu_B \quad\mbox{ for all } I\subseteq[n],
\end{equation}
where by convention $\kappa_{\emptyset}=\mu_{\emptyset}$. Let $\mathcal
{K}_{T}=f_{\mu\kappa}(\mathcal{C}_{n})$. Note that for any $I\subseteq
[n]$ such that $|I|\leq3$, $\kappa_I=\mu_I$. In particular $\mathcal
{K}_{T}$ is contained in the subspace of $\mathbb{R}^{n}\times\mathbb
{R}^{2^{n}}$ given by
\[
\kappa_{\emptyset}=1,\quad\kappa_{1}=\cdots=\kappa_{n}=0
\]
The map $f_{\mu\kappa}:\mathcal{C}_{n}\rightarrow\mathcal{K}_{T}$ is a
polynomial isomorphism with a polynomial inverse $f_{\kappa\mu}$. It
therefore gives a change of coordinates to a coordinate system on
$\mathcal{K}_{T}$ given by $\lambda_{1},\ldots, \lambda_{n}$ and $\kappa
_{I}$ for $|I|\geq2$. The exact form of the inverse map is given by
the M\"{o}bius inversion formula (c.f. \cite[Section
3.2]{pwz-2010-identifiability})
\begin{equation}\label{eq:muinkappa}
\mu_{I}=\sum_{\pi\in\Pi_{{T}(I)}}\prod_{B\in\pi} \kappa_B\quad\mbox{
for all } I\subseteq[n], |I|\geq2.
\end{equation}
Note that after restriction to $\Delta_{2^{n}-1}$, $f_{p\lambda}(\Delta
_{2^{n}-1})$ and $\mathcal{C}_{n}$ respectively all $f_{p\lambda}$,
$f_{\lambda\mu}$ and $f_{\mu\kappa}$ are polynomial maps with
polynomial inverses (c.f. \cite[Appendix A]{pwz-2010-identifiability}).
This therefore implies that there is a polynomial isomorphism between
$\Delta_{2^{n}-1}$ and $\mathcal{K}_{T}$.

\section{Proofs}\label{sec:proofs}

\begin{proof}[Proof of Proposition \ref{lem:semi_tripod}]
By Remark \ref{rem:rootings} $\mathcal{M}_{3}$ does not depend on the
rooting. Therefore, we can assume that $T$ is rooted in $h$. In this
case Proposition \ref{prop:monomial} implies that $\mathcal
{M}_{3}^{\kappa}$ is given by $\lambda_{i}=\frac{1}{2}(1-\bar{\mu
}_{i})$ for $i=1,2,3$ and
%
\begin{equation}\label{eq:star}
\begin{array}{l}
\mu_{ij}=\dfrac{1}{4}(1-\bar{\mu}_{h}^2)\eta_{h,i}\eta_{h,j}\mbox{ for
all } i\neq j\in\{1,2,3\} \mbox{ and }\\[.3cm]
\mu_{123}=\dfrac{1}{4}(1-\bar{\mu}_{h}^2)\bar{\mu}_{h} \eta_{h,1}\eta
_{h,2}\eta_{h,3},
\end{array}
\end{equation}
subject to constraints in (\ref{eq:constraints}).

Denote the subset of $\mathcal{K}_{T}$ given by constraints (i),(ii) by
$\mathcal{M}$. We need to show that $\mathcal{M}=\mathcal{M}_{3}^{\kappa
}$. First, we prove that $\mathcal{M}_{3}^{\kappa}\subseteq\mathcal
{M}$. Let $K=\psi_{T}(\omega)$ for some $\omega\in\Omega_{T}$ with
coordinates given by $\bar{\mu}_{h}$ and $\bar{\mu}_{i}$, $\eta_{h,i}$
for $i=1,2,3$. We consider two cases. Either $(1-\bar{\mu}_{h}^{2})\eta
_{h,1}\eta_{h,2}\eta_{h,3}$ is zero or not. In the first case $\mu
_{123}=0$ and at least two covariances vanish and hence (i) holds.

Now we show that if $(1-\bar{\mu}_{h}^{2})\eta_{h,1}\eta_{h,2}\eta
_{h,3}\neq0$ then (ii) holds. From (\ref{eq:star})
%
\begin{equation}\label{eq:prod-mus}
\mu_{12}\mu_{13}\mu_{23}\quad=\quad\left(\frac{1}{4}(1-\bar{\mu
}_{h}^{2})\right)^{3}(\eta_{h,1}\eta_{h,2}\eta_{h,3})^{2}\quad>\quad0.
\end{equation}
To show that $K$ satisfies (\ref{eq:series-ineq-tripod2}) we can simply
substitute for the corresponding moments using (\ref{eq:star}). After
trivial reductions we then obtain that
\[
|\eta_{h,i}|\pm\bar{\mu}_{h}\eta_{h,i}\leq(1\pm\bar{\mu}_{i}),
\]
which is equivalent to (\ref{eq:constraints}). Therefore, since by
hypothesis (\ref{eq:constraints}) holds, we also have that $\mathcal
{M}_{3}^{\kappa}\subseteq\mathcal{M}$.

To show $\mathcal{M}\subseteq{\mathcal{M}_{3}^{\kappa}}$ we prove that
for $K\in\mathcal{M}$ a parameter $\omega$ in (\ref{eq:star}) exists
which satisfies the constraints defining $\Omega_{T}$ and $K=\psi
_{T}(\omega)$. Let $P$ be the probability distribution corresponding to
$K$. First, consider the points satisfying (i). If all three
covariances vanish for this point then taking $\eta_{h,1}=\eta
_{h,2}=\eta_{h,3}=0$ and $\bar{\mu}_{h}^{2}=1$ we obtain a valid choice
of parameters in (\ref{eq:star}) and their values satisfy (\ref
{eq:constraints}). When one covariance is non-zero, say $\mu_{12}\neq
0$, then, if a choice of parameters exists it must satisfy $\bar{\mu
}_{h}^{2}\neq1$, $\eta_{h,1},\eta_{h,2}\neq0$ and $\eta_{h,3}=0$.
Such a choice of parameters will exist if we can ensure that $\mu
_{12}=(1-\bar{\mu}_{h}^{2})\eta_{h,1}\eta_{h,2}$. This follows from
\cite[Corollary 2]{gilula1979svd} which states that if only $\mu
_{12}\neq0$ then there always exists a choice of parameters for model
$X_{1}\indep X_{2}|H$, where $H$ is hidden.

Consider now case (ii). Since $\mu_{12}\mu_{13}\mu_{23}>0$ then in
particular ${\rm Det} \,P>0$. Set $\bar{\mu}_{h}^{2}=\frac{\mu
_{123}^{2}}{{{\rm Det} \,P}}$ and $\eta_{h,i}^{2}=\frac{{{\rm Det} \,
P}}{\mu_{jk}^{2}}$ for $i=1,2,3$. It follows that $(\frac{1}{4}(1-\bar
{\mu}_{h}^{2}))^{2}\eta_{h,i}^{2}\eta_{h,j}^{2}=\mu_{ij}^{2}$ for
$i,j=1,2,3$ and $(\frac{1}{4}(1-\bar{\mu}_{h}^{2}))^{2} \bar{\mu
}_{h}^{2} \eta_{h,1}^{2}\eta_{h,2}^{2}\eta_{h,3}^{2}=\mu_{123}^{2}$.
This coincides with (\ref{eq:star}) modulo the sign. It can be easily
shown that $\mu_{12}\mu_{13}\mu_{23}>0$ implies that there exist a
choice of signs for $\eta_{h,i}$ for $i=1,2,3$ such that
\[
\frac{1}{4}(1-\bar{\mu}_{h}^{2})\eta_{h,i}\eta_{h,j}=\mu_{ij}
\]
for all $1\leq i<j\leq3$ as in (\ref{eq:star}). For example set ${\rm
sgn}(\eta_{h,i})={\rm sgn}(\mu_{jk})$ and use the fact that, by our
assumption, ${\rm sgn}(\mu_{ij})={\rm sgn}(\mu_{ik}){\rm sgn}(\mu
_{jk})$. This choice of signs already determines the sign of $\bar{\mu
}_{h}$ so that
\[
\frac{1}{4}(1-\bar{\mu}_{h}^{2}) \bar{\mu}_{h} \eta_{h,1}\eta_{h,2}\eta
_{h,3}=\mu_{123}
\]
holds.

It remains to show that parameters set in this way satisfy the
constraints defining $\Omega_{T}$. First note that since $0<4\mu_{12}\mu
_{13}\mu_{23}\leq{\rm Det} \,P$ then $\bar{\mu}_{h}^{2}\in(0,1)$ as
required. From \cite[Appendix D]{pwz-2010-identifiability} we know that
if $(\eta_{h,1},\eta_{h,2},\eta_{h,3},\bar{\mu}_{h})$ is one choice of
parameters then there exists only one alternative choice and it is
$(-\eta_{h,1},-\eta_{h,2},-\eta_{h,3},-\bar{\mu}_{h})$. For a fixed
$i=1,2,3$ it is easily checked that $(\eta_{h,i},\bar{\mu}_{h})$
satisfies (\ref{eq:constraints}) if and only if $(-\eta_{h,i},-\bar{\mu
}_{h})$ does. Therefore, we can assume that $\eta_{h,i}=\frac{\sqrt
{{\rm Det} P}}{|\mu_{jk}|}>0$. In this case $\bar{\mu}_{h}={\rm sgn}(\mu
_{jk})\frac{\mu_{123}}{\sqrt{{\rm Det}P}}$. It follows that (\ref
{eq:constraints}) is satisfied if and only if (\ref
{eq:series-ineq-tripod2}) holds.
\end{proof}

\begin{proof}[Proof of Lemma \ref{lem:inequalities2}]
First assume that the map $s_{0}:E\rightarrow\{-1,1\}$, given in the
statement of the lemma, exists. This induces a map $s:V\times
V\rightarrow\{-1,1\}$ such that $s(k,l)=\prod_{(u,v)\in
E(kl)}s_{0}(u,v)$. For any triple $i,j,k$ there exists a unique inner
node $h$ which is the intersection of all three paths between $i,j,k$.
By the above equation the choice of signs for all $(u,v)\in E$ gives
$s(i,h),s(j,h)$ and $s(k,h)$. Since $s(i,j)=s(i,h)s(j,h)$ and the same
for the two other pairs, we get that
$s(i,j)s(i,k)s(j,k)=s^2(i,h)s^2(j,h)s^2(k,h)=1$ and the result follows
since by construction $\sigma(i,j)=s(i,j)$ for all $i,j\in[n]$.

Now we prove the converse implication. Whenever there is a path $E(uv)$
in $T$ such that all its inner nodes have degree two then a sign
assignment satisfying (\ref{eq:sij}) exists if and only if there exists
a sign assignment for the same tree but with $E(uv)$ contracted to a
single edge $(u,v)$. Hence we can assume that the degree of each inner
node is at least three.

We use an inductive argument with respect to number of hidden nodes.
First we will show that the theorem is true for trees with one inner
node (star trees) denoted by $h$. In this case we will use induction
with respect to number of leaves. It can easily be checked directly
that the theorem is true for the tripod tree. Assume it works for all
star trees with $k\leq m-1$ leaves and let $T$ be a star tree with $m$
leaves. By assumption for any three leaves $i,j,k$: $\sigma(i,j)\sigma
(i,k)\sigma(j,k)=1$. If we consider a subtree with $(1,h)$ deleted then
by induction assumption we can find a consistent choice of signs for
all remaining edges. A choice of a sign for $(1,h)$ consistent with
(\ref{eq:sij}) exists if for all $i\geq2$ $\sigma
(1,i)=s_{0}{(1,h)}s_{0}(i,h)$. This is true if either $\sigma
(1,i)s_{0}(i,h)=1$ for all $i$ or $\sigma(1,i)s_{0}(i,h)=-1$ for all
$i$. Assume it is not true, i.e. there exist two leaves $i,j$ such that
$\sigma(1,i)s_{0}(i,h)=1$ and $\sigma(1,j)s_{0}(j,h)=-1$. Then in
particular since $\sigma(i,j)=s_{0}(i,h)s_{0}(j,h)$ we would have that
$\sigma(1,i)\sigma(1,j)\sigma(i,j)=-1$ which contradicts our assumption.

If the number of the inner nodes is greater than one then pick an inner
node $h$ adjacent to exactly one inner node. Let $h'$ be the inner node
adjacent to $h$ and let $I$ be a subset of leaves which are adjacent to
$h$. Choose one $i\in I$ and consider a subtree $T'$ obtained by
removing all leaves in $I$ and the incident edges apart from the node
$i$ and the edge $(h,i)$. By the induction, since $h$ has degree two in
the resulting subtree, we can find signs for all edges of $T'$. Set
$s_{0}(h,h')=1$ then $s_{0}(h,i)=s(h',i)$ which identifies
$s_{0}(h,i)$. Similarly it can be showed that there exists a choice of
signs for all remaining edges $(i',h)$. The result follows since the
choice of $i\in I$ was arbitrary.
\end{proof}

\section{The proof of the main theorem}\label{sec:proof}

Let $K\in\mathcal{K}_{T}$ have coordinates given by $\lambda_{i}$ for
$i=1,\ldots,n$ and $\kappa_{I}$ for $I\subseteq[n]$ such that $|I|\geq
2$. Let $K^{J}$, $J\subseteq[n]$, denote the projection onto the
coordinates given by $\lambda_{i}$ for $i\in J$ and $\kappa_{I}$,
$I\subseteq J$, $|I|\geq2$. Directly from the definition of $\mathcal
{M}_{T}$ it follows that $K\in\mathcal{M}_{T}^{\kappa}$ if and only if
$K^{I}\in\mathcal{M}_{T(I)}^{\kappa}$ for all $I\subseteq[n]$.

Let $\mathcal{M}$ denote the subset of $\mathcal{K}_{T}$ defined by
constraints in (C1)-(C4). We need to show that $\mathcal{M}=\mathcal
{M}_{T}^{\kappa}$. We divide the proof into series of lemmas.

\begin{lem}\label{lem:inclusion1}
The inclusion $\mathcal{M}_{T}^{\kappa}\subseteq\mathcal{M}$ holds.
\end{lem}
\begin{proof}
Since the rooting is not relevant by Remark \ref{rem:rootings}, we
choose an arbitrary inner node as the root node. Let $K\in\mathcal
{M}_{T}^{\kappa}$ and hence $K=\psi_{T}(\omega)$ for some $\omega\in
\Omega_{T}$.

To show that the equations in (C1) hold let $A|B$ be an edge split and
let $e=(w,w')$ be the edge inducing this split. By $T\setminus e$ we
denote the graph obtained from $T$ by removing the edge $e$. We assume
that $w$ lies in the same connected component of $T\setminus e$ as $A$
and $w'$ lies in the second component of $T\setminus e$. For every
non-empty $I\subseteq A$ and $J\subseteq B$ from Proposition \ref{prop:monomial}
\begin{eqnarray*}
\kappa_{IJ}&=&\frac{1}{4}(1-\bar{\mu}_{r(IJ)}^{2})\prod_{v\in{\rm
int}(V(Iw'))}\bar{\mu}_{v}^{{\rm deg}(v)-2}\prod_{v\in{\rm
int}(V(Jw))}\bar{\mu}_{v}^{{\rm deg}(v)-2}\\
& &{} \cdot\eta_{w,w'}\prod_{(u,v)\in E(Iw)}\eta_{u,v}\prod_{(u,v)\in
E(Jw')}\eta_{u,v}.
\end{eqnarray*}
From this it easily follows that for any non-empty
$I_{1},I_{2}\subseteq A$ and $J_{1},J_{2}\subseteq B$, $\kappa
_{I_{1}J_{1}}\kappa_{I_{2}J_{2}}-\kappa_{I_{1}J_{2}}\kappa
_{I_{2}J_{1}}=0$ if and only if
%
\begin{equation}\label{eq:roots-compensate}
(1-\mu_{r(I_{1}J_{1})}^{2})(1-\mu_{r(I_{2}J_{2})}^{2})=(1-\mu
_{r(I_{1}J_{2})}^{2})(1-\mu_{r(I_{2}J_{1})}^{2}).
\end{equation}
To show that (\ref{eq:roots-compensate}) is always true, we consider
two cases: either $r(AB)\in V(Aw)$ or $r(AB)\in V(Bw')$. If $r(AB)\in
V(Aw)$ then $r(I_{1}J_{1})=r(I_{1}w)$, $r(I_{1}J_{2})=r(I_{1}w)$,
$r(I_{2}J_{1})=r(I_{2}w)$ and $r(I_{2}J_{2})=r(I_{2}w)$. Hence in this
case (\ref{eq:roots-compensate}) holds. The case $r(AB)\in V(Bw')$
follows by symmetry. Therefore the equations in (C1) always hold.

To show that $K$ satisfies (C2) consider the projection $K^{ijk}$ for
each $i,j,k\in[n]$. By \cite[Corollary 2.2]{pwz-2010-identifiability}
$\mathcal{M}_{T(ijk)}^{\kappa}$ is equal to the tripod tree model.
Since $K^{ijk}\in\mathcal{M}_{T(ijk)}^{\kappa}$ then, by Proposition
\ref{lem:semi_tripod}, (C2) must hold. To show that $K$ satisfies (C3)
let $i,j\in[n]$ be such that $\mu_{ij}=0$. Let $I\subseteq[n]$ be
such that $i,j\in I$ and assume that $\kappa_{I}(\omega)\neq0$. Then
by (\ref{eq:kappa_def_general}) in particular $\mu_{r(I)}^{2}\neq1$
and $\eta_{u,v}\neq0$ for all $(u,v)\in E(I)$. By \cite[Remark
4.3]{pwz-2010-identifiability} this implies in particular that $\bar{\mu
}_{r(ij)}^{2}\neq1$. From this, again by (\ref{eq:kappa_def_general}),
it follows that $\mu_{ij}\neq0$ and we get a contradiction. Hence if
$\mu_{ij}=0$ then $\kappa_{I}=0$ for all $I$ such that $i,j\in I$.

To show that $K$ satisfies (C4) let $i,j,k,l\in[n]$ be the four leaves
mentioned in the condition. Let $u$ and $v$ be two inner nodes such
that $u$ separates $i$ from $j$, $v$ separates $k$ from $l$ and $\{u,v\}
$ separates $\{i,j\}$ from $\{k,l\}$. In other words $u$, $v$ are the
only inner nodes of degree three in $T(ijkl)$. By \cite[Lemma
2.1]{pwz-2010-identifiability}, $T(ijkl)$ gives the same model as the
quartet tree with four leaves $i,j,k,l$ and two inner nodes $u$, $v$.
Moreover, by Remark \ref{rem:rootings}, $\mathcal{M}_{T(ijkl)}$ does
not depend on the rooting so we can assume that the tree is rooted in
$u$. Since $K^{ijkl}\in\mathcal{M}_{T(ijkl)}$ then for some parameter choices
\[
\mu_{ik}=\frac{1}{4}(1-\bar{\mu}_{u}^{2})\eta_{u,i}\eta_{u,v}\eta
_{v,k},\quad\mu_{jl}=\frac{1}{4}(1-\bar{\mu}_{u}^{2})\eta_{u,j}\eta
_{u,v}\eta_{v,l}
\]
\[
\mu_{ijk}=\frac{1}{4}(1-\bar{\mu}_{u}^{2})\bar{\mu}_{u}\eta_{u,i}\eta
_{u,j}\eta_{u,v}\eta_{v,k},\quad\mu_{ikl}=\frac{1}{4}(1-\bar{\mu
}_{u}^{2})\bar{\mu}_{v}\eta_{u,i}\eta_{u,v}\eta_{v,k}\eta_{v,l}.
\]
Substitute these equations into (C4). There are then two cases to
consider: $\mu_{uv}\geq0$, $\mu_{uv}<0$. Laborious but elementary
algebra shows that the condition in (C4) is equivalent to (\ref
{eq:constraints}) applied to $(1-\bar{\mu}_{u}^{2})\eta_{u,v}$ and
hence (C4) holds by definition. Consequently $\mathcal{M}_{T}^{\kappa
}\subseteq\mathcal{M}$.
\end{proof}

To show the opposite inclusion is a bit more complicated. We consider
two separate cases. Let $K\in\mathcal{M}$. We construct a point $\omega
_{0}\in\mathbb{R}^{|V|+|E|}$ such that $\omega_{0}\in\Omega_{T}$ and
$\psi_{T}(\omega_{0})=K$, i.e. $\omega_{0}$ is such that, for all
$I\subseteq[n]$ such that $|I|\geq2$, $\kappa_{I}$ can be written in
terms of the parameters in $\omega_{0}$ as in (\ref{eq:kappa_def_general}).
\begin{lem}\label{lem:case1}
Let $K$ be such that $\mu_{ij}\neq0$ for all $i,j\in[n]$. If $K\in
\mathcal{M}$ then $K\in\mathcal{M}_{T}^{\kappa}$.
\end{lem}
\begin{proof}
We set squares of values of all the parameters in terms of the observed
moments using \cite[Corollary 5.5]{pwz-2010-identifiability}. We will
show that the equations in (\ref{eq:kappa_def_general}) must hold for
their absolute values. We will then need to ensure there is at least
one assignment of signs for a set of parameters such that all equations
in (\ref{eq:kappa_def_general}) hold exactly. Finally, we will show
that the parameter vector $\omega_{0}$ defined in this way lies in
$\Omega_{T}$.

For each inner node $h$ of $T$ let $i,j,k\in[n]$ be any three leaves
separated by $h$ in $T$. By (C2) we have that $\mu_{ij}\mu_{ik}\mu
_{jk}> 0$ and hence also that ${\rm Det}P^{ijk}>0$. Now set
%
\begin{equation}\label{eq:paramsq1}
(\bar{\mu}^{0}_{h})^{2}=\frac{\mu_{ijk}^{2}}{{{\rm Det} P^{ijk}}}.
\end{equation}
We show that (C1), which $K$ satisfies by assumption, implies that the
value of $(\bar{\mu}_{h}^{0})^{2}$ does not depend on the choice of
$i,j,k$. It suffices to show that if $k$ is replaced by another leaf
$k'$ such that $i,j,k'$ are separated by $h$ in $T$ then $\frac{\mu
_{ijk}^{2}}{{{\rm Det} P^{ijk}}}=\frac{\mu_{ijk'}^{2}}{{{\rm Det}
P^{ijk'}}}$. Since $h$ has degree three in $T$ then there exists an
edge $e\in E$ inducing a split $A|B$ such that $i,j\in A$ and $k,k'\in
B$. From (C1) it follows that
%
\begin{equation}\label{eq:mupom1}
\mu_{ik}\mu_{jk'}=\mu_{ik'}\mu_{jk},\quad\mu_{ijk}\mu_{ik'}=\mu
_{ijk'}\mu_{ik},\quad\mu_{ijk}\mu_{jk'}=\mu_{ijk'}\mu_{jk}
\end{equation}
and consequently
%
\begin{equation}\label{eq:mupom2}
{\rm Det} P^{ijk}\mu_{ij}\mu_{ik'}\mu_{jk'}={\rm Det} P^{ijk'}\mu
_{ij}\mu_{ik}\mu_{jk}
\end{equation}
which implies that
\[
\frac{\mu_{ijk}^{2}}{{{\rm Det} P^{ijk}}}=\frac{\mu_{ijk}^{2}\mu_{ij}\mu
_{ik'}\mu_{jk'}}{{{\rm Det} P^{ijk}}\mu_{ij}\mu_{ik'}\mu_{jk'}}=\frac
{\mu_{ijk'}^{2}\mu_{ij}\mu_{ik}\mu_{jk}}{{{\rm Det} P^{ijk'}}\mu_{ij}\mu
_{ik}\mu_{jk}}=\frac{\mu_{ijk'}^{2}}{{{\rm Det} P^{ijk'}}}
\]
as required.

For terminal edges $(v,i)$ of $T$ such that $i\in[n]$, let $j,k\in
[n]$ be any two leaves of $T$ such that $v$ separates $i$, $j$, $k$. Set
%
\begin{equation}\label{eq:paramsq2}
(\eta_{v,i}^{0})^{2}=\frac{{{\rm Det} P^{ijk}}}{\mu_{jk}^{2}}.
\end{equation}
As in the previous case it is straightforward to check that, given
(C1), this value does not depend on the choice of $j,k$. For example,
if instead of $k$ we have $k'$ and $v$ separates $i,j,k'$ in $T$ then
there exists an edge split such that $\{i,j\}$ and $\{k,k'\}$ are in
different blocks. By (\ref{eq:mupom1}), we can show that
\[
\frac{{{\rm Det} P^{ijk}}}{\mu_{jk}^{2}}=\frac{\mu_{ik}{{\rm Det}
P^{ijk}}}{\mu_{ik'}\mu_{jk'}\mu_{jk}}=\frac{{{\rm Det} P^{ijk'}}}{\mu
_{jk'}^{2}}.
\]

For inner edges $(u,v)\in E$ let $i,j,k,l\in[n]$ be any four leaves
such that $u$ separates $i$ from $j$, $v$ separates $k$ from $l$ and $\{
u,v\}$ separates $\{i,j\}$ from $\{k,l\}$. Set
%
\begin{equation}\label{eq:inner_cor2}(\eta_{u,v}^{0})^{2}=\frac{\mu
_{il}^{2}}{\mu_{ij}^{2}}\frac{{\rm Det} P^{ijk}}{{\rm Det} \,P^{ikl}}
\end{equation}
which is well-defined since $\mu_{ij}^{2}$ and ${\rm Det} P^{ikl}$ are
strictly positive.
We now show that this value does not depend on the choice of $i,j,k,l$.
By symmetry it suffices to show that we obtain the same value if
instead of $l$ we took another leaf $l'$ such that $u,v$ are the only
degree three nodes in $T(ijkl')$. Since $v$ has degree three then there
must exist an inner edge separating $i,j,k$ from $l,l'$. From (C1) it
follows that
\[
\mu_{il'}\mu_{kl'}{\rm Det}P^{ikl}=\mu_{il}\mu_{kl}{\rm Det} P^{ikl'},
\quad\mu_{il}\mu_{kl'}=\mu_{il'}\mu_{kl}
\]
and hence
\[
\frac{\mu_{il}^{2}}{\mu_{ij}^{2}}\frac{{\rm Det} P^{ijk}}{{\rm Det} \,
P^{ikl}} = \frac{\mu_{il'}\mu_{kl'}}{\mu_{il'}\mu_{kl'}}\frac{\mu
_{il}^{2}}{\mu_{ij}^{2}}\frac{{\rm Det} P^{ijk}}{{\rm Det} \,P^{ikl}}
=\frac{\mu_{il'}^{2}}{\mu_{ij}^{2}}\frac{{\rm Det} P^{ijk}}{{\rm Det} \,
P^{ikl'}}
\]
as required.

We now show that with the choice of parameters satisfying (\ref
{eq:paramsq1}), (\ref{eq:paramsq2}) and (\ref{eq:inner_cor2}) the
modulus of equations in (\ref{eq:kappa_def_general}) hold. First
consider the case $I=\{i,j\}$. Label the inner nodes of $E(ij)$ by
$v_{1},\ldots, v_{k}$ beginning from the node adjacent to $i$. For each
$s=1,\ldots,k$ let $i_{s}$ denote a leaf such that $v_{s}$ separates
$i,j,i_{s}$ in $T$. By Remark \ref{rem:rootings}, we can choose any
rooting. We assume that the root $r(ij)$ of this path is in $v_{1}$. We
now proceed to check that
%
\begin{eqnarray}\label{eq:nistwo}
\mu_{ij}^{2}&=&\left(\frac{1}{4}(1-(\bar{\mu}^{0}_{r(ij)})^{2})\right
)^{2}\prod_{(u,v)\in E(ij)} (\eta_{u,v}^{0})^{2}\\
\nonumber&=& \left(\frac{1}{4}(1-(\bar{\mu}^{0}_{r(ij)})^{2})\right
)^{2}(\eta_{v_{1},u}^{0})^{2} \left(\prod_{s=2}^{k} (\eta
_{v_{s-1},v_{s}}^{0})^{2}\right) (\eta_{v_{k},v}^{0})^{2}.
\end{eqnarray}
Since $v_{1}$ separates $i,j,i_{1}$ by construction, from (\ref
{eq:paramsq1}) we therefore have
\[
\frac{1}{4}(1-(\bar{\mu}_{v_{1}}^{0})^{2})=\frac{\mu_{ij}\mu_{ii_{1}}\mu
_{ji_{1}}}{{\rm Det}(P^{iji_{1}})}.
\]
Now substitute this equation and all the set values in (\ref
{eq:paramsq2}), (\ref{eq:inner_cor2}) into the right hand side of (\ref
{eq:nistwo}). Use the fact that $v_{k}$ separates $i,j,i_{k}$ in $T$
and $i_{s-1},i_{s}$ are the only degree three nodes in
$T(ii_{s-1}ji_{s})$. Since $(v_{1},i)$ and $(v_{k},j)$ are the only
terminal edges we obtain
%
\begin{eqnarray}
&& \left(\frac{\mu_{ij}\mu_{ii_{1}}\mu_{j i_{1}}}{{\rm
Det}(P^{iji_{1}})}\right)^{2}\cdot\frac{{{\rm Det} P^{iji_{1}}}}{\mu
_{ji_{1}}^{2}}\cdot\left(\prod_{s=2}^{k} \frac{\mu_{ii_{s}}^{2}}{\mu
_{ii_{s-1}}^{2}}{\frac{{\rm Det} P^{iji_{s-1}}}{{\rm Det} \,
P^{iji_{s}}}}\right)\cdot\frac{{{\rm Det}
P^{iji_{k}}}}{\mu_{ji_{k}}^{2}}\quad
\end{eqnarray}
It can now be checked that all the expressions with hyperdeterminants
cancel out and the formula reduces to $\mu_{ij}^{2}$ as required.

Now we need to show that for every $I=\{i,j,k\}$
%
\begin{equation}\label{eq:pompompom3}
\mu_{ijk}^{2}=\left(\frac{1}{4}(1-\bar{\mu}^{0}_{r(ijk)})^{2}) \right
)^{2}(\bar{\mu}_{w}^{0})^{2}\prod_{(u,v)\in E(ijk)} (\eta_{u,v}^{0})^{2},
\end{equation}
where by $w$ we denote the node separating $i$, $j$ and $k$. Assume
that $T(ijk)$ is rooted somewhere on the path between $i$ and $j$.
Using (\ref{eq:nistwo}) the right hand side of (\ref{eq:pompompom3})
can be rewritten as
%
\begin{equation}\label{eq:pompompom4}
\mu_{ij}^{2}(\bar{\mu}_{w}^{0})^{2}\prod_{(u,v)\in E(wk)}(\eta_{u,v}^{0})^{2}.
\end{equation}
Number the degree three nodes in $E(wk)$ by $v_{1},\ldots, v_{l}$ and
let $i_{s}$ denote a leaf such that the inner nodes of $T(ijk i_{s})$
of degree three are exactly $v_{s-1}$ and $v_{s}$, where $v_{0}=w$. By
an exactly analogous argument as in the case above we obtain
%
\begin{eqnarray}\label{eq:bigexpmuijk}
&&\nonumber\prod_{(u,v)\in E(wk)}(\eta_{u,v}^{0})^{2}\\
&&\qquad = \frac{\mu_{i i_{1}}^{2}}{\mu_{ij}^{2}}{\frac
{{\rm Det} P^{ijk}}{{\rm Det} P^{iki_{1}}}}\cdot\left(\prod
_{s=2}^{l}\frac{\mu_{i_{s-1} i_{s}}^{2}}{\mu_{i_{s-2}i_{s-1}}^{2}}{\frac
{{\rm Det} P^{i_{s-2}i_{s-1}k}}{{\rm Det} P^{i_{s-1}i_{s}k}}}\right)
\frac{{{\rm Det} P^{i_{l-1}i_{l}k}}}{\mu_{i_{l-1}i_{l}}^{2}},\qquad
\end{eqnarray}
where $i_{0}=i$. It can be easily checked that all the
hyperdeterminants apart from the term ${{\rm Det} P^{ijk}}$ cancel out.
Moreover, all the covariances apart from the term $\mu_{ij}^{-2}$
cancel out as well. Hence (\ref{eq:bigexpmuijk}) is equal to $\frac
{{{\rm Det} P^{ijk}}}{\mu_{ij}^{2}}$. Now, by using the definition of
$(\bar{\mu}^{0}_{w})^{2}$ in (\ref{eq:paramsq1}), it can be easily
checked that (\ref{eq:pompompom4}) is equal to $\mu_{ijk}^{2}$ as required.

So far we have confirmed only that the squares of parameters in $\omega
_{0}$ satisfy required equations at least for the tree cumulants up to
the third order. Next, we show that there exists a consistent choice of
signs for these parameters such that the equations are satisfied
exactly. Let $\sigma(i,j)={\rm sgn}(\mu_{ij})$. Since by assumption $\mu
_{ij}\neq0$ for all $i,j\in[n]$ then the conditions in (C2) imply that
$\sigma(i,j)\sigma(i,k)\sigma(j,k)=1$ for all triples $i,j,k\in[n]$.
Hence by Lemma \ref{lem:inequalities2} there exists a choice
$s_{0}(u,v)\in\{-1,+1\}$ for all $(u,v)\in E$ such that $\sigma
(i,j)=\prod_{(u,v)\in E(ij)} s_{0}(u,v)$ for all $i,j\in[n]$. For any
two nodes $k,l\in V$ we define $s(k,l)=\prod_{(u,v)\in
E(kl)}s_{0}(u,v)$. A choice of signs for the parameters can be obtained
as follows: For each edge $(u,v)\in E$ we set ${\rm sgn}(\eta
_{u,v}^{0})=s_{0}(u,v)$ and, for each inner node $v$, set ${\rm
sgn}(\bar{\mu}_{v}^{0})={\rm sgn}(\mu_{ijk})s(v,i)s(v,j)s(v,k)$ where
$i,j,k$ are any three leaves of $T$ separated by $v$.

Assume now that the choice of the signs of the parameters, induced by
$s_{0}(u,v)$ for $(u,v)\in E$, has been made. This choice of signs gives
%
\begin{equation}\label{eq:param1}
\bar{\mu}_{v}^{0}=s(v,i)s(v,j)s(v,k)\frac{\mu_{ijk}}{\sqrt{{{\rm Det}
P^{ijk}}}},
\end{equation}
%
\begin{equation}\label{eq:param2}
\eta_{v,i}^{0}=s(v,i)\frac{\sqrt{{{\rm Det} P^{ijk}}}}{|\mu_{jk}|},
\end{equation}
%
\begin{equation}\label{eq:param3}
\eta_{u,v}^{0}=s_{0}(u,v)\left|\frac{\mu_{il}}{\mu_{ij}}\right|\sqrt
{\frac{{\rm Det} P^{ijk}}{{\rm Det} \,P^{ikl}}}.
\end{equation}
Note that, in particular, with this choice of signs ${\rm sgn}(\eta
_{u,v}^{0})=s_{0}(u,v)$ for all $(u,v)\in E$ and ${\rm sgn}(\bar{\mu
}_{v}^{0})={\rm sgn}(\mu_{ijk})\prod_{(u,v)\in E(ijk)}s_{0}(u,v)$.
Since (\ref{eq:nistwo}) holds, it follows that
\[
|\mu_{ij}|=\frac{1}{4}(1-(\bar{\mu}^{0}_{r(ij)})^{2})\prod_{(u,v)\in
E(ij)} |\eta_{u,v}^{0}|.
\]
Now multiply both sides by $s(i,j)=\prod_{(u,v)\in E(ij)}s_{0}(u,v)$ to get
%
\begin{eqnarray}
\nonumber\mu_{ij}=s(i,j)|\mu_{ij}|&=&\frac{1}{4}(1-(\bar{\mu
}^{0}_{r(ij)})^{2})\prod_{(u,v)\in E(ij)} s_{0}(u,v)|\eta_{u,v}^{0}|\\
&=&\frac{1}{4}(1-(\bar{\mu}^{0}_{r(ij)})^{2})\prod_{(u,v)\in E(ij)} \eta
_{u,v}^{0}.
\end{eqnarray}
Similarly, from (\ref{eq:pompompom3}), we have that
\begin{equation*}
|\mu_{ijk}|=\frac{1}{4}(1-(\bar{\mu}^{0}_{r(ijk)})^{2})|\bar{\mu
}_{w}^{0}|\prod_{(u,v)\in E(ijk)} |\eta_{u,v}^{0}|.
\end{equation*}
Multiply both sides by ${\rm sgn}(\mu_{ijk})$ and use the fact that
$(\prod_{(u,v)\in E(ijk)} s_{0}(u,v))^{2}=1$ to get
\begin{eqnarray*}
\mu_{ijk}&=&\frac{1}{4}(1-(\bar{\mu}^{0}_{r(ijk)})^{2})\left(|\bar{\mu
}_{w}^{0}|\,{\rm sgn}(\mu_{ijk})\prod_{(u,v)\in E(ijk)}
s_{0}(u,v)\right)\\
&&{} \cdot\prod_{(u,v)\in E(ijk)} s_{0}(u,v)|\eta_{u,v}^{0}|\\
&=&\frac{1}{4}(1-(\bar{\mu}^{0}_{r(ijk)})^{2}) \bar{\mu}_{w}^{0}\prod
_{(u,v)\in E(ijk)} \eta_{u,v}^{0}
\end{eqnarray*}
as desired.

We now show (\ref{eq:kappa_def_general}) for $|I|\geq4$ by induction.
Let $(u,v)\in E$ be any edge splitting $I$ into two subsets $I_{1}$ and
$I_{2}$ such that $|I_{1}|,|I_{2}|\geq2$ and $u$ is the node closer to
$I_{1}$. Let $i\in I_{1}$ and $j\in I_{2}$ then, by (C1),
\[
\kappa_{I_{1}I_{2}}=\frac{\kappa_{I_{1}j}\kappa_{i I_{2}}}{\kappa_{ij}}.
\]
By induction we can assume that $\kappa_{I_{1}j}$, $\kappa_{i I_{2}}$
and $\kappa_{ij}$ have form as in (\ref{eq:kappa_def_general}). Moreover,
\[
\frac{\prod_{(u,v)\in E(iI_{2})}\eta_{u,v}\prod_{(u,v)\in E(I_{1}j)}\eta
_{u,v}}{\prod_{(u,v)\in E(ij)}\eta_{u,v}}=\prod_{(u,v)\in E(I)}\eta_{u,v},
\]
\[
\prod_{h\in N(i I_{2})}\bar{\mu}_{h}^{\deg h -2}=\prod_{h\in N(v
I_{2})}\bar{\mu}_{h}^{\deg h -2},
\]
\[
\prod_{h\in N(I_{1}j)}\bar{\mu}_{h}^{\deg h -2}=\prod_{h\in N(
I_{1}u)}\bar{\mu}_{h}^{\deg h -2}.
\]
Using this we can write
%
\begin{equation}\label{eq:pompom5}
\kappa_{I_{1}I_{2}}=\frac{1}{4}\frac{(1-\bar{\mu
}_{r(iI_{2})}^{2})(1-\bar{\mu}_{r(I_{1}j)}^{2})}{(1-\bar{\mu
}_{r(ij)}^{2})}\prod_{h\in N(I)}\bar{\mu}_{h}^{\deg h -2}\prod
_{(u,v)\in E(I)}\eta_{u,v}.
\end{equation}
The root of $T(I)$ is either in $T(I_{1}u)$ or in $T(v I_{2})$. In the
first case $r(I_{1}j)=r(I)$ and $r(i I_{2})=r(ij)$. In the second case
$r(I_{1}j)=r(ij)$ and $r(i I_{2})=r(I)$. Hence in both cases
\[
\frac{(1-\bar{\mu}_{r(iI_{2})}^{2})(1-\bar{\mu
}_{r(I_{1}j)}^{2})}{(1-\bar{\mu}_{r(ij)}^{2})}=(1-\bar{\mu}_{r(I)}^{2})
\]
and (\ref{eq:pompom5}) has the required form given by (\ref
{eq:muinkappa}). It follows that $K=\psi_{T}(\omega_{0})$.

It now remains to show that the parameters defined in (\ref
{eq:param1}), (\ref{eq:param2}) and (\ref{eq:param3}) define a
parameter vector $\omega_{0}$ which lies in $\Omega_{T}$. Since, by
(C2), $\mu_{ijk}^{2}\leq{\rm Det} P^{ijk}$ for all $i,j,k\in[n]$ for
all inner nodes $h$ we have $\bar{\mu}_{h}^{0}\in[-1,1]$ as required.
For a terminal edge $(v,i)$ consider the marginal model induced by
$T(ijk)$, where $j,k$ are any two leaves such that $v$ separates
$i,j,k$ in $T$. From Proposition \ref{lem:semi_tripod} constraints (C2)
and (C3) imply that $\eta_{v,i}$ is a valid parameter. To show that
(\ref{eq:param3}) satisfies (\ref{eq:constraints}) write
\[
(1\pm\bar{\mu}_{u}^{0})\eta_{u,v}^{0}=\left(1\pm s(u,i)s(u,j)s(u,k)\frac
{\mu_{ijk}}{\sqrt{{\rm Det }P^{ijk}}}\right)s(u,v)\left|\frac{\mu
_{il}}{\mu_{ij}}\right|\sqrt{\frac{{\rm Det }P^{ijk}}{{\rm Det }P^{ikl}}}.
\]
Now substitute this together with the expressions for $\bar{\mu
}_{u}^{0}$ and $\bar{\mu}_{v}^{0}$, given by (\ref{eq:param1}), into
(\ref{eq:constraints}). First assume $s(u,v)=1$. Then $s(u,k)=s(v,k)$,
$s(v,i)=s(u,i)$ and (\ref{eq:constraints}) becomes
\[
\left(\sqrt{{\rm Det }P^{ijk}}\pm s(u,i)\mu_{ijk}\right)\left|\frac{\mu
_{il}}{\mu_{ij}}\right|\leq\left(\sqrt{{\rm Det }P^{ikl}}\pm s(v,l)\mu
_{ikl}\right).
\]
By multiplying both sides by a positive expression $|\mu_{jl}|(\sqrt
{{\rm Det }P^{ijk}}\mp s(u,i)\mu_{ijk})$ we obtain
\[
{4\mu_{ik}^{2}\mu_{jl}^{2}}\leq\left(\sqrt{{\rm Det }P^{ijk}}\pm
s(u,l)\mu_{jl}{\mu_{ijk}}\right)\left(\sqrt{{\rm Det }P^{ikl}}\mp
s(v,l){\mu_{ikl}}\right).
\]
However, $s(u,l)=s(v,l)$ hence this is satisfied by (C5). It is easily
calculated that the case $s(u,v)=-1$ leads to the same constraint. This
finishes the proof of Lemma \ref{lem:case1}.
\end{proof}

\begin{lem}\label{lem:case2}
The inclusion $\mathcal{M}\subseteq\mathcal{M}^{\kappa}_{T}$ holds.
\end{lem}
\begin{proof}
Let ${K}\in\mathcal{M}$ be a tree cumulant and let ${\Sigma}=[{\mu
}_{ij}]\in\mathbb{R}^{n\times n}$ be the matrix of all covariances
between the leaves. We say that an edge $e\in E$ is \textit{isolated
relative to} ${K}$ if ${\mu}_{ij}=0$ for all $i,j\in[n]$ such that
$e\in E(ij)$. By $\widehat{E}\subseteq E$ we denote the set of all
edges of $T$ which are isolated relative to ${K}$. By $\widehat
{T}=(V,E\setminus\widehat{E})$ we denote the forest obtained from $T$
by removing edges in $\widehat{E}$ and we call it the $K$-forest. We
define relations on $\widehat{E}$ and $E\setminus\widehat{E}$. For two
edges $e,e'$ with either $\{e,e'\}\subset\widehat{E}$ or $\{e,e'\}
\subset E\setminus\widehat{E}$ write $e\sim e'$ if either $e=e'$ or $e$
and $e'$ are adjacent and all the edges that are incident with both $e$
and $e'$ are isolated relative to $K$. Let us now take the transitive
closure of $\sim$ restricted to pairs of edges in $\widehat{E}$ to form
an equivalence relation on $\widehat{E}$. This transitive closure is
constructed as follows. Consider a graph with nodes representing
elements of $\widehat{E}$ and put an edge between $e,e'$ whenever
$e\sim e'$. Then the equivalence classes correspond to connected
components of this graph. Similarly, take the transitive closure of
$\sim$ restricted to the pairs of edges in $E\setminus\widehat{E}$ to
form an equivalence relation in $E\setminus\widehat{E}$. We will let
$[\widehat{E}]$ and $[E\setminus\widehat{E}]$ denote the set of
equivalence classes of $\widehat{E}$ and $E\setminus\widehat{E}$
respectively (for details see \cite[Section 5]{pwz-2010-identifiability}).

Again we show that there exists $\omega_{0}\in\Omega_{T}$ such that
$\psi_{T}(\omega_{0})=K$. Set $\eta_{u,v}^{0}=0$ for all $(u,v)\in
\widehat{E}$ and $\bar{\mu}_{v}^{0}=0$ for all inner nodes of $T$ with
degree zero in $\widehat{T}$. It then follows that $(1\pm\bar{\mu
}_{u})\eta_{u,v}=0$ satisfies (\ref{eq:constraints}) for all $(u,v)\in
\widehat{E}$ and $\bar{\mu}_{v}^{0}\in[-1,1]$ for all $v\in\widehat{V}$
and hence these parameters satisfy constraints defining $\Omega_{T}$.
If $I\subseteq[n]$ is such that $E(I)\cap\widehat{E}\neq\emptyset$
then $\kappa_{I}=0$ by (C3). Hence in this case we can assert that
\[
\kappa_{I}=\frac{1}{4}(1-(\bar{\mu}_{r(I)}^{0})^{2})\prod_{v\in
N(I)}(\bar{\mu}_{v}^{0})^{\deg(v)-2}\prod_{(u,v)\in E({I})}\eta_{u,v}^{0}
\]
simply because both sides of this equation are zero. By \cite[Remark
5.2 (iv)]{pwz-2010-identifiability} every connected component of
$\widehat{T}$ is a subtree which is either an inner node or a tree with
the set of leaves contained in $[n]$. Denote the connected subtrees
which are not inner nodes by $T_{1},\ldots, T_{k}$ and their sets of
leaves by $[n_{l}]$ for $l=1,\ldots,k$. For every $l=1,\ldots, k$ and
all $i,j\in[n_{l}]$ we have that $\mu_{ij}\neq0$. Hence for each
$T_{l}$ applying Lemma \ref{lem:case1} we have $K^{[n_{l}]}\in\mathcal
{M}_{T_{l}}$. If $I\subseteq[n]$ is such that $E(I)\cap\widehat{E}=
\emptyset$ then $I\subseteq[n_{l}]$ for some $l=1,\ldots, k$. Since
$K^{[n_{l}]}\in\mathcal{M}_{T_{l}}$ then there exists a choice of
parameters such that $\kappa_{I}$ can be written as (\ref
{eq:kappa_def_general}). Therefore $K\in\mathcal{M}_{T}$ and we are done.
\end{proof}

The proof that $\mathcal{M}=\mathcal{M}_{T}^{\kappa}$ follows from
Lemma \ref{lem:inclusion1} and Lemma \ref{lem:case2}. It suffices to
show that, given that all covariances are non-zero, the only
constraints of $\mathcal{M}$ involving only second order moments are
(\ref{eq:suffic-ineq2}). In the formulation of the main result the only
such constraints are all the equations in (C1) involving only
covariances and the positivity constraints in (C2). By the four-point
condition (c.f. (\ref{eq:suffic-ineq1})) the inequalities
\[
\min\left\{\left(\frac{\mu_{ik}\mu_{jl}}{\mu_{ij}\mu_{kl}}\right
)^{2},\left(\frac{\mu_{il}\mu_{jk}}{\mu_{ij}\mu_{kl}} \right)^{2}\right
\}\leq1
\]
for all not necessarily distinct $i,j,k,l\in[n]$ uniquely define the
underlying tree metric and hence they are equivalent to all the
equations in (C1) involving only second order moments. The inequalities
\[
\min\left\{\frac{\mu_{ik}\mu_{jl}}{\mu_{ij}\mu_{kl}},\frac{\mu_{il}\mu
_{jk}}{\mu_{ij}\mu_{kl}} \right\}\geq0
\]
are equivalent to $\mu_{ij}\mu_{ik}\mu_{jk}\geq0$ for all $i,j,k\in
[n]$. However, the two above sets of inequalities are exactly
equivalent to (\ref{eq:suffic-ineq2}). \hfill$\square$

%
%
%

\section{Phylogenetic invariants}\label{sec:invariants}

In a seminal paper Allman and Rhodes \cite{allman2008pia} identified
equations defining the general Markov $\mathcal{M}_{T}$ in the case
when $T$ is a trivalent tree. In this section we relate their results
to ours. To introduce their main theorem we need the following definition.
\begin{defn}\label{def:flat}
Let $X=(X_1,\ldots,X_n)$ be a vector of binary random variables and let
$P=(p_\gamma)_{\gamma\in\{0,1\}^n}$ be a $2\times\ldots\times2$ table
of the joint distribution of $X$. Let $A|B$ form a split of $[n]$. Then
the \textit{flattening} of $P$ induced by the split is a matrix
\[
P_{A|B}=[p_{\alpha\beta}], \quad\alpha\in\{0,1\}^{|A|},\beta=\{0,1\}^{|B|},
\]
where $p_{\alpha\beta}=\mathbb{P}(X_A=\alpha,X_B=\beta)$. Let $T=(V,E)$
be a tree. In particular, for edge partitions the induced flattening is
called an \textit{edge flattening} and we denote it by $P_e$, where
$e\in E$ is the edge inducing the split.
\end{defn}

Note that whenever we implicitly use some order on coordinates indexed
by $\{0,1\}$-sequences we always mean the order induced by the
lexicographic order on $\{0,1\}$-sequences such that ${0\cdots
00}>0\cdots01>\ldots>1\cdots11$. This gives in particular the
ordering of rows and columns of flattenings.
\begin{thm}[Allman, Rhodes \cite{allman2008pia}]\label{th:minors}
Let $T^{r}$ be a trivalent tree rooted in $r$ and $\mathcal{M}_{T}$ be
the general Markov model on $T^{r}$ as defined by (\ref{eq:p_albar2}).
Then the smallest algebraic variety, i.e. a subset of a real space
defined by a finite set of polynomial equations, containing the general
Markov model, is defined by vanishing of all $3\times3$-minors of all
the edge flattenings of $T^{r}$ together with the trivial polynomial
equation $\sum_\alpha p_\alpha=1$.
\end{thm}
Note that the result includes the case of the tripod tree model since
in this case each edge flattening of the joint probability table is a
$2\times4$ table so there are no $3\times3$ minors and hence there
are no non-trivial polynomials vanishing on the model.

Just as we defined edge flattenings of probability tables we can also
define edge flattenings of $(\kappa_I)_{I\subseteq[n]}$ where $\kappa
_{\emptyset}=1$ and $\kappa_{i}=0$ for all $i\in[n]$ (c.f. Appendix
\ref{app:change}). Let $e$ be an edge of $T$ inducing a split $A|B\in\Pi
_T$ such that $|A|=r$, $|B|=n-r$. Then $\widehat{N}_e$ is a $2^r\times
2^{n-r}$ matrix such that for any two subsets $I\subseteq A$,
$J\subseteq B$ the element of $\widehat{N}_e$ corresponding to the
$I$-th row and the $J$-th column is $\kappa_{IJ}$. Let $N_{e}$ denote
its submatrix given by removing the column and the row corresponding to
empty subsets of $A$ and $B$. Here the labeling for the rows and
columns is induced by the ordering of the rows and columns for $P_e$
(c.f. Definition \ref{def:flat}), i.e. all the subsets of $A$ and $B$
are coded as $\{0,1\}$-vectors and we introduce the lexicographic order
on the vectors with the vector of ones being the last one.

The following result allows us to rephrase the equations in Theorem \ref
{th:minors} in terms of our new coordinates.
\begin{prop}\label{prop:flat}
Let $T=(V,E)$ be a tree and let $P$ be a probability distribution of a
vector $X=(X_1,\ldots, X_n)$ of binary variables represented by the
leaves of $T$. If $e\in E$ is an edge of $T$ inducing a split
$A_{1}|A_{2}$ then ${\rm rank} (P_e)=2$ if and only if ${\rm rank}(N_e)=1$.
\end{prop}
\begin{proof}
Let $P_e=[p_{\alpha\beta}]$ be the matrix induced by a split $A_1|A_2$.
We will show that ${\rm rank }(P_{e})={\rm rank}(D_{e})$ where
$D_{e}=[d_{IJ}]$ is a block diagonal matrix with $1$ as the first
$1\times1$ block (i.e. $d_{\emptyset\emptyset}=1$, $d_{\emptyset
J}=0$, $d_{I \emptyset}=0$ for all $I\subseteq A_{1}$, $J\subseteq
A_{2}$) and the matrix $N_e$ as the second block. It will then follow
that ${\rm rank}(P_e)=2$ if and only if ${\rm rank}(N_e)=1$.

First note that the flattening matrix $P_e$ can be transformed to the
flattening of the non-central moments just by adding rows and columns
according to (\ref{eq:lambda_in_p}) and then to the flattening of the
central moments $M_e=[\mu_{IJ}]$ such that $I\subseteq A_1$,
$J\subseteq A_2$ using (\ref{eq:central}). It therefore suffices to
show that ${\rm rank }(M_{e})={\rm rank}(D_{e})$.

Let $I\subseteq A_1$, $J\subseteq A_2$. Then for each $\pi\in\Pi
_{T({IJ})}$ there is at most one block containing elements from both
$I$ and $J$. For if this were not so then removing $e$ would increase
the number of blocks in $\pi$ by more than one which is not possible.
Denote this block by $(I'J')$ where $I'\subseteq I$, $J'\subseteq J$.
Note that by construction we have either both $I',J'$ are empty sets if
$\pi\geq A_1|A_2$ in $\Pi_{T(IJ)}$ or both $I',J'\neq\emptyset$
otherwise. We can rewrite (\ref{eq:muinkappa}) as
%
\begin{equation}\label{eq:equat_pom1}
\mu_{IJ}=\sum_{\pi\in\Pi_{{T}(IJ)}}\left(\kappa_{I'J'}\prod_{I\supseteq
B\in\pi}\kappa_B \prod_{J\supseteq B\in\pi}\kappa_B\right).
\end{equation}
We have $d_{I'J'}=\kappa_{I'J'}$ and it can be further rewritten as
\[
\mu_{IJ}=\sum_{I'\subseteq I}\sum_{J'\subseteq J}u_{I I'} d_{I'J'} v_{J'J}
\]
where $u_{I I'}=\sum_{\pi\in\Pi_{T}(I\setminus I')}\prod_{B\in\pi
}\kappa_{B}$ and $v_{J'J}=\sum_{\pi\in\Pi_{T}(J\setminus J')}\prod
_{B\in\pi}\kappa_{B}$. Setting $u_{I I'}=0$ for $I'\nsubseteq I$,
$v_{J'J}=0$ for $J'\nsubseteq J$ we can write these coefficients in
terms of a lower triangular matrix $U$ and an upper triangular matrix
$V$. Since by construction $u_{I I}=1$ for all $I\subseteq A_1$ and
$v_{JJ}=1$ for all $J\subseteq A_2$ we have $\det U=\det V=1$.
Therefore, $M_{e}$ has the same rank as $D_{e}$.
\end{proof}

The proposition shows that the vanishing of all $3\times3$ minors of
all the edge flattenings of $P$ and the trivial invariant $\sum p_\alpha
=1$ are together equivalent to the vanishing all $2\times2$ minors of
all edge flattenings of $\kappa=(\kappa_I)_{I\in[n]_{\geq2}}$. An
immediate corollary follows which gives the equations in (C1) in
Theorem (\ref{th:parameters0}).
\begin{cor}\label{prop:equations}
Let $T=(V,E)$ be a trivalent tree. Then the smallest algebraic variety
containing $\mathcal{M}_T^{\kappa}$ is defined by the following set of
equations. For each split $A|B$ induced by an edge consider any four
(not necessarily disjoint) nonempty sets $I_1,I_2\subseteq A$,
$J_1,J_2\subseteq B$ and the induced equation $\kappa_{I_1 J_1}\kappa
_{I_2 J_2}-\kappa_{I_1 J_2}\kappa_{I_2 J_1}=0$.
\end{cor}

In \cite{eriksson2007uip} Eriksson noted that some of the invariants
usually prove to be better in discriminating between different tree
topologies than the others. His simulations showed that the invariants
related to the four-point condition were especially powerful. The
binary case we consider in this paper can give some partial
understanding of why this might be so. Here, the invariants related to
the four-point condition are the only ones which involve second order
moments (c.f. Section \ref{sec:metrics}). Moreover, the estimates of
the higher-order moments (or cumulants) are sensitive to outliers and
their variance generally grows with the order of the moment. Let $\hat
{\mu}$ be a sample estimator of the central moments $\mu$ and let $f$
be one of the polynomials in Theorem \ref{prop:equations} but expressed
in terms of the central moments. Then using the delta method we have
\[
{\rm Var}(f(\hat{\mu}))\simeq\nabla f(\mu)^t {\rm Var}(\hat{\mu})
\nabla f(\mu).
\]
Consequently, in this loose sense at least, the higher the order of the
central moments (or equivalently the higher the order of the tree
cumulants) the higher the variability of we might expect the invariant
to exhibit (see \cite[Section 4.5]{mccullagh1987tms}).

}


\end{document}